\documentclass[12pt,reqno]{amsart}

\input{preamble.tex}

\title{Quasi-likelihood analysis of an ergodic diffusion plus noise}
\author[S. H. Nakakita]{Shogo H. Nakakita$^{1}$}
\author[M. Uchida]{Masayuki Uchida$^{1,2}$}
\address{$^{1}$Graduate School of Engineering Science, Osaka University}
\address{$^{2}$Center for Mathematical Modeling and Data Science, Osaka University}

\begin{document}
\onehalfspacing
\maketitle

\begin{abstract}
	We consider adaptive maximum-likelihood-type estimators and adaptive Bayes-type ones for discretely observed ergodic diffusion processes with observation noise whose variance is constant.
	The quasi-likelihood functions for the diffusion and drift parameters are introduced and the polynomial-type large deviation inequalities for those quasi-likelihoods are shown to see the convergence of moments for those estimators.
\end{abstract}

\section{Introduction}

% setting of latent processes
We consider a $d$-dimensional ergodic diffusion process defined by the following stochastic differential equation such that
\begin{align*}
\mathrm{d}X_{t}=b\left(X_{t},\beta\right)\mathrm{d}t+a\left(X_{t},\alpha\right)\mathrm{d}w_{t},\ X_{0}=x_{0},
\end{align*}
where $\left\{w_{t}\right\}_{t\ge 0}$ is an $r$-dimensional Wiener process, $x_{0}$ is a random variable independent of $\left\{w_{t}\right\}_{t\ge 0}$, $\alpha\in\Theta_{1}$ and $\beta\in\Theta_{2}$ are unknown parameters, $\Theta_{1}\subset\R^{m_{1}}$ and $\Theta_{2}\subset\R^{m_{2}}$ are bounded, open and convex sets in $\R^{m_{i}}$ admitting Sobolev's inequalities for embedding $W^{1,p}\left(\Theta_{i}\right)\hookrightarrow C\left(\overline{\Theta}_{i}\right)$ for $i=1,2$, $\theta^{\star}=\left(\alpha^{\star},\beta^{\star}\right)$ is the true value of the parameter, and $a:\R^{d}\times \Theta_{1}\to \R^{d}\otimes \R^{r}$ and $b:\R^{d}\times \Theta_{2}\to \R^{d}$ are known functions.

% setting of observation
A matter of interest is to estimate the parameter $\theta=\left(\alpha,\beta\right)$ with partial and indirect observation of $\left\{X_{t}\right\}_{t\ge0}$: the observation is discretised and contaminated by exogenous noise. The sequence of observation $\left\{Y_{ih_{n}}\right\}_{i=0,\ldots,n}$, which our parametric estimation is based on, is defined as
\begin{align*}
	Y_{ih_{n}}=X_{ih_{n}}+\Lambda^{1/2}\varepsilon_{ih_{n}},\ i=0,\ldots,n,
\end{align*}
where $h_{n}>0$ is the discretisation step such that $h_{n}\to0$ and $T_{n}=nh_{n}\to\infty$, $\left\{\varepsilon_{ih_{n}}\right\}_{i=0,\ldots,n}$ is an i.i.d.\ sequence of random variables independent of $\left\{w_{t}\right\}_{t\ge 0}$ and $x_0$ such that $\mathbf{E}_{\theta^{\star}}\left[\varepsilon_{ih_{n}}\right]=0$ and $\mathrm{Var}_{\theta^{\star}}\left(\varepsilon_{ih_{n}}\right)=I_{d}$ where $I_{m}$ is the identity matrix in $\R^{m}\otimes \R^{m}$ for every $m\in\mathbf{N}$, and $\Lambda\in\R^{d}\otimes\R^{d}$ is a positive semi-definite matrix which is the variance of noise term. We also assume that the half vectorisation of $\Lambda$ has bounded, open and convex parameter space $\Theta_{\varepsilon}$, and let us denote $\Xi:=\Theta_{\varepsilon}\times \Theta_{1}\times \Theta_{2}$. We also notate the true parameter of $\Lambda$ as $\Lambda_{\star}$, its half vectorisation as $\theta_{\varepsilon}^{\star}=\mathrm{vech}\Lambda_{\star}$, and $\vartheta^{\star}=\left(\theta_{\varepsilon}^{\star},\alpha^{\star},\beta^{\star}\right)$. That is to say, our interest is on parametric inference for an ergodic diffusion with long-term and high-frequency noised observation. One concrete example is the wind velocity data provided by \citet{NWTC} whose observation is contaminated by exogenous noise with statistical significance according to the test for noise detection \citep{Nakakita-Uchida-2018a}.

\begin{figure}[h]
	\begin{subfigure}{.5\textwidth}
		\centering
		\includegraphics[bb=0 0 720 480,width=7cm]{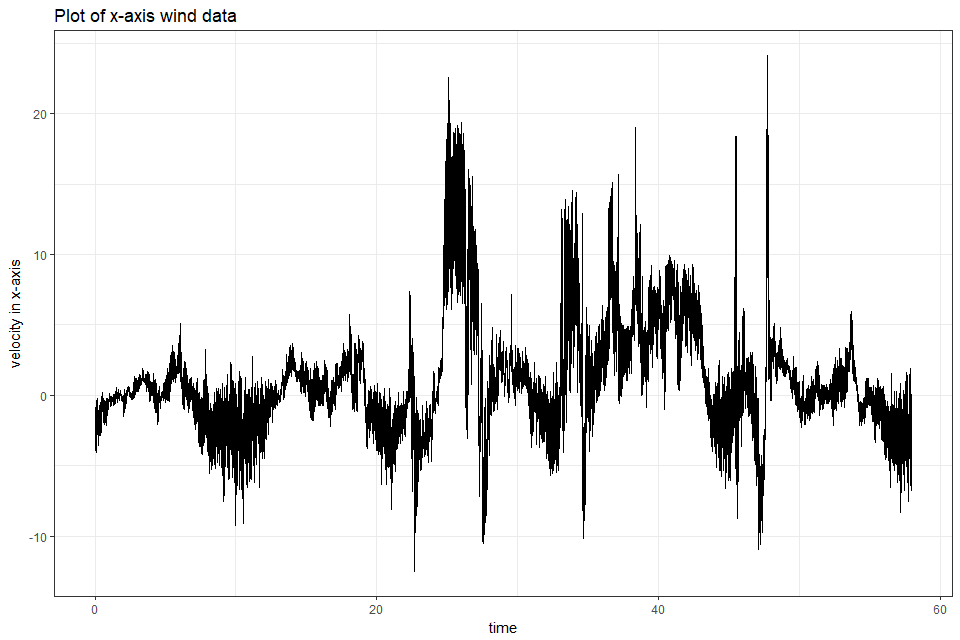}
	\end{subfigure}%
	\begin{subfigure}{.5\textwidth}
		\centering
		\includegraphics[bb=0 0 720 480,width=7cm]{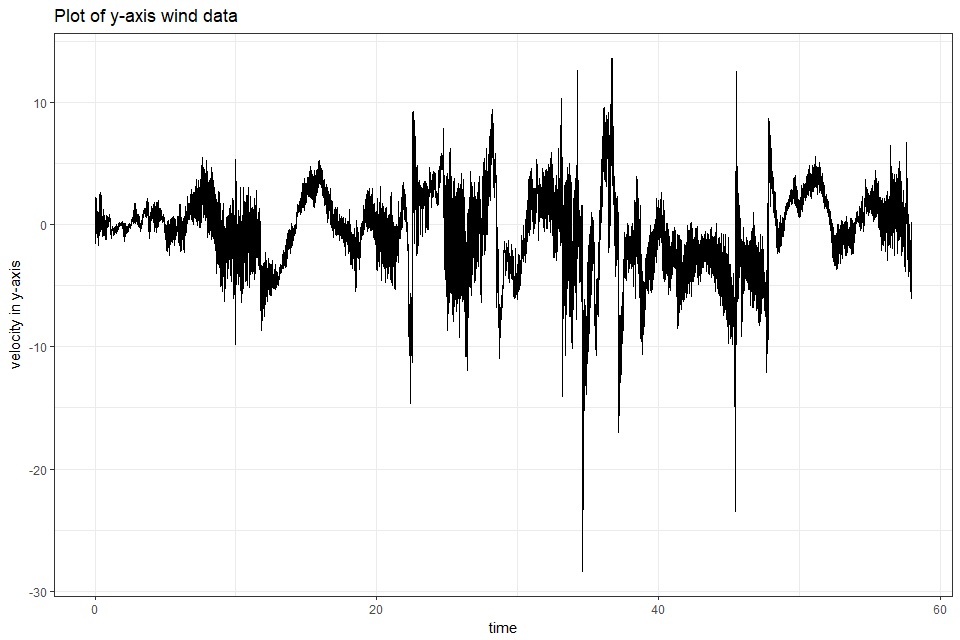}
	\end{subfigure}
	\caption{plot of wind velocity labelled Sonic x (left) and y (right) (119M) at the M5 tower from 00:00:00 on 1st July, 2017 to 20:00:00 on 5th July, 2017 with 0.05-second resolution \citep{NWTC}}\label{MetData}
\end{figure}

As the existent discussion, \cite{Nakakita-Uchida-2018a} propose the following estimator $\hat{\Lambda}_{n}$, $\hat{\alpha}_{n}$ and $\hat{\beta}_{n}$ such that
\begin{align*}
	&\hat{\Lambda}_{n}=\frac{1}{2n}\sum_{i=0}^{n-1}\left(Y_{\left(i+1\right)h_{n}}-Y_{ih_{n}}\right)^{\otimes 2},\\
	&\mathbb{H}_{1,n}^{\tau}\left(\hat{\alpha}_{n};\hat{\Lambda}_{n}\right)
	=\sup_{\alpha\in\Theta_{1}}\mathbb{H}_{1,n}^{\tau}\left(\alpha;\hat{\Lambda}_{n}\right),\\
	&\mathbb{H}_{2,n}\left(\hat{\beta}_{n};\hat{\alpha}_{n}\right)
	=\sup_{\beta\in\Theta_{2}}\mathbb{H}_{2,n}\left(\beta;\hat{\alpha}_{n}\right),
\end{align*}
where for every matrix $A$, $A^{T}$ is the transpose of $A$ and $A^{\otimes 2}=AA^{T}$, $\mathbb{H}_{1,n}^{\tau}$ and $\mathbb{H}_{2,n}$ are the adaptive quasi-likelihood functions of $\alpha$ and $\beta$ respectively defined in Section 3, $\tau\in\left(1,2\right]$ is a tuning parameter, and \citet{Nakakita-Uchida-2018a} show these estimators are asymptotically normal and especially the drift one is asymptotically efficient. To obtain the convergence rates of the estimators, it is necessary to see the composition of the quasi-likelihood functions. Both of them are function of local means of observation defined as
\begin{align*}
	\lm{Y}{j}=\frac{1}{p_{n}}\sum_{i=0}^{p_{n}-1}Y_{j\Delta_{n}+ih_{n}},\ j=0,\ldots,k_{n}-1,
\end{align*}
where $k_{n}$ is the number of partition given for observation, $p_{n}$ is that of observation in each partition, $\Delta_{n}=p_{n}h_{n}$ is the time interval which each partition has, and note that these parameters have the properties $k_{n}\to\infty$, $p_{n}\to\infty$ and $\Delta_{n}\to0$. 
Intuitively speaking, $k_{n}$ and $\Delta_{n}$ correspond to $n$ and $h_{n}$ in the observation scheme without exogenous noise, and divergence of $p_{n}$ works to eliminate the influence of noise by law of large numbers. Hence it should be also easy to understand that we have the asymptotic normality with the convergence rates $\sqrt{k_{n}}$ and $\sqrt{T_{n}}$ for $\alpha$ and $\beta$; that is,
\begin{align*}
	\left[\sqrt{k_{n}}\left(\hat{\alpha}_{n}-\alpha^{\star}\right),\sqrt{T_{n}}\left(\hat{\beta}_{n}-\beta^{\star}\right)\right]\rightarrow^{d}\xi,
\end{align*}
where $\xi$ is an $\left(m_{1}+m_{2}\right)$-dimensional Gaussian distribution with zero-mean.

The statistical inference for diffusion processes with discretised observation has been investigated in these decades: see \cite{Florens-Zmirou-1989}, \cite{Yoshida-1992}, \cite{Bibby-Sorensen-1995}, \cite{Kessler-1995,Kessler-1997}. In practice, it is necessary to argue whether exogenous noise exists in observation, and it has been pointed out that the observational noise, known as microstructure noise, certainly exists in high-frequency financial data which is one of the major disciplines where statistics for diffusion processes is applied. Inference for diffusions under such the noisy and discretised observation in fixed time interval $\left[0,1\right]$ is discussed by \cite{Jacod-et-al-2009}, and also \cite{Favetto-2014,Favetto-2016} examine same problem as our study and shows simultaneous ML-type estimation has consistency under the situation where the variance of noise is unknown and asymptotic normality under the situation where the variance is known. As mentioned above, \cite{Nakakita-Uchida-2018a} propose adaptive ML-type estimation which has asymptotic normality even if we do not know the variance of noise, and test for noise detection which succeeds in showing the real data \cite{NWTC} which is contaminated by observational noise.

Our study aims at polynomial type large deviation inequalities for statistical random fields and construction of the estimators with not only asymptotic normality as shown in \citet{Nakakita-Uchida-2018a} but also a certain type of convergence of moments. Asymptotic normality is well-known as one of the hopeful properties that estimators are expected to have; for instance, \cite{Nakakita-Uchida-2018b} utilise this result to compose likelihood-ratio-type statistics and related ones for parametric test and proves the convergence in distribution to a $\chi^2$-distribution under null hypothesis and consistency of the test under alternative one. However, it is also known that asymptotic normality is not sufficient to develop some discussion requiring convergence of moments such as information criterion. In concrete terms, it is necessary to shows the convergence of moments such as for every $f\in C\left(\R^{m_{1}}\times \R^{m_{2}}\right)$ with at most polynomial growth and adaptive ML-type estimator $\hat{\alpha}_{n}$ and $\hat{\beta}_{n}$,
\begin{align*}
	\mathbf{E}_{\vartheta^{\star}}
	\left[f\left(\sqrt{k_{n}}\left(\hat{\alpha}_{n}-\alpha^{\star}\right),\sqrt{T_{n}}\left(\hat{\beta}_{n}-\beta^{\star}\right)\right)\right]\to \mathbf{E}_{\theta^{\star}}\left[f\left(\xi\right)\right].
\end{align*}
This property is stronger than mere asymptotic normality since if we take $f$ as a bounded and continuous function, then indeed asymptotic normality follows.

To see the convergence of moments for adaptive ML-type estimator, we can utilise polynomial-type large deviation inequalities (PLDI) and quasi-likelihood analysis (QLA) proposed by \cite{Yoshida-2011} which have been widely used to discuss convergence of moments of not only ML-type estimation but also Bayes-type one in statistical inference for continuous-time stochastic processes. This approach is developed from the exponential-type large deviation and likelihood analysis introduced by \cite{Ibragimov-Hasminskii-1972,Ibragimov-Hasminskii-1973,Ibragimov-Hasminskii-1981}, and the polynomial-type one discussed by \cite{Kutoyants-1984,Kutoyants-1994,Kutoyants-2004}. \cite{Yoshida-2011} itself discusses convergence of moments in adaptive maximum-likelihood-type estimation, simultaneous Bayes-type one, and adaptive Bayes-type one for ergodic diffusions with $nh_{n}\to\infty$ and $nh_{n}^{2}\to0$. \cite{Uchida-Yoshida-2012,Uchida-Yoshida-2014} examine the same problem for adaptive ML-type and adaptive Bayes-type estimation for ergodic diffusions with more relaxed condition: $nh_{n}\to\infty$ and $nh_{n}^{p}\to0$ for some $p\ge2$. \cite{Ogihara-Yoshida-2011} study convergence of moments for parametric estimators against ergodic jump-diffusion processes in the scheme of $nh_{n}\to\infty$ and $nh_{n}^{2}\to0$. Other than diffusion processes or jump-diffusions, \cite{Clinet-Yoshida-2017} show PLDI for the quasi-likelihood function for ergodic point processes and the convergence of moments for the corresponding ML-type and Bayes-type estimators. As the applications of these discussions, \cite{Uchida-2010} composes AIC-type information criterion for ergodic diffusion processes, and \cite{Eguchi-Masuda-2018} propose BIC-type one for local-asymptotic quadratic statistical experiments including some schemes for diffusion processes. In this paper, we develop QLA for our ergodic diffusion plus noise model and propose the adaptive Bayes-type estimators of both drift and volatility parameters. Furthermore, we show the convergence of moments of both the adaptive ML-type estimators and the adaptive Bayes-type estimators for the ergodic diffusion plus noise model. Note that Bayes-type estimation itself is important to deal with non-linearity of parameters and multimodality of quasi-likelihood functions which sometimes appear in statistics for diffusion processes. In particular, the hybrid type estimators with initial Bayes-type estimators are considered for diffusion type processes, see \citet{Kamatani-Uchida-2015, Kaino-Uchida-2018a, Kaino-Uchida-2018b}, and references therein. Moreover, as an application of the Bayes-type estimation proposed in this paper, \citet{Kaino-et-al-2018} study the hybrid estimators with initial Bayes-type estimators for our ergodic diffusion plus noise model and give an example and simulation results of the hybrid estimator.

\section{Notation and assumption}

We set the following notations.
\begin{itemize}
	\item For every matrix $A$, $A^{T}$ is the transpose of $A$, and $A^{\otimes 2}:=AA^{T}$.
	\item For every set of matrices $A$ and $B$ whose dimensions coincide, $A\left[B\right]:=\mathrm{tr}\left(AB^{T}\right)$. Moreover, for any $m\in\mathbf{N}$, $A\in\R^{m}\otimes\R^{m}$ and $u,v\in\R^{m}$, $A\left[u,v\right]:=v^{T}Au$.
	\item Let us denote the $\ell$-th element of any vector $v$ as $v^{\left(\ell\right)}$ and $\left(\ell_{1},\ell_{2}\right)$-th one of any matrix $A$ as $A^{\left(\ell_{1},\ell_{2}\right)}$.
	\item For any vector $v$ and any matrix $A$, $\left|v\right|:=\sqrt{\mathrm{tr}\left(v^{T}v\right)}$ and $\left\|A\right\|:=\sqrt{\mathrm{tr}\left(A^{T}A\right)}$.
	\item For every $p>0$, $\left\|\cdot\right\|_{p}$ is the $L^{p}\left(P_{\theta^{\star}}\right)$-norm.
	\item $A\left(x,\alpha\right):=a\left(x,\alpha\right)^{\otimes 2}$, $a\left(x\right):=a\left(x,\alpha^{\star}\right)$, $A\left(x\right):=A\left(x,\alpha^{\star}\right)$ and $b\left(x\right):=b\left(x,\beta^{\star}\right)$.
	\item For given $\tau\in\left(1,2\right]$, $p_{n}:=h_{n}^{-1/\tau}$, $\Delta_{n}:=p_{n}h_{n}$, and $k_{n}:=n/p_{n}$, and we define the sequence of local means such that
	\begin{align*}
		\lm{Z}{j}=\frac{1}{p_{n}}\sum_{i=0}^{p_{n}-1}Z_{j\Delta_n+ih_{n}},\ j=0,\cdots,k_{n}-1,
	\end{align*}
	where $\left\{Z_{ih_{n}}\right\}_{i=0,\ldots,n}$ indicates an arbitrary sequence defined on the mesh
	$\left\{ih_{n}\right\}_{i=0,\ldots,n}$ such as $\left\{Y_{ih_{n}}\right\}_{i=0,\ldots,n}$, $\left\{X_{ih_{n}}\right\}_{i=0,\ldots,n}$ and $\left\{\varepsilon_{ih_{n}}\right\}_{i=0,\ldots,n}$.
\end{itemize}

\begin{remark}
Since the observation is masked by the exogenous noise, it should be transformed to obtain the undermined process $\left\{X_{t}\right\}_{t\ge0}$. As illustrated by \cite{Nakakita-Uchida-2018a}, the sequence $\left\{\lm{Y}{j}\right\}_{j=0,\ldots,k_{n}-1}$ can extract the state of the latent process $\left\{X_{t}\right\}_{t\ge0}$  in the sense of the statement of Lemma \ref{ApproxLM}.
\end{remark}
\begin{itemize}
	\item $\mathcal{G}_{t}:=\sigma\left(x_{0},w_{s}:s\le t\right)$, $\mathcal{G}_{j,i}^{n}:=\mathcal{G}_{j\Delta_{n}+ih_{n}}$,
	$\mathcal{G}_{j}^{n}:=\mathcal{G}_{j,0}^{n}$, 
	$\mathcal{A}_{j,i}^{n}:=\sigma\left(\varepsilon_{\ell h_{n}}:\ell \le jp_{n}+i-1\right)$, $\mathcal{A}_{j}^{n}:=\mathcal{A}_{j,0}^{n}$, $\mathcal{H}_{j,i}^{n}:=\mathcal{G}_{j,i}^{n}\vee \mathcal{A}_{j,i}^{n}$
	and $\mathcal{H}_{j}^{n}:=\mathcal{H}_{j,0}^{n}$.
	\item We define the real-valued function as for $l_{1},l_{2},l_{3},l_{4}=1,\ldots,d$:
	\begin{align*}
	&V\left((l_1,l_2),(l_3,l_4)\right)\\
	&:=\sum_{k=1}^{d}\left(\Lambda_{\star}^{1/2}\right)^{(l_1,k)}\left(\Lambda_{\star}^{1/2}\right)^{(l_2,k)}\left(\Lambda_{\star}^{1/2}\right)^{(l_3,k)}\left(\Lambda_{\star}^{1/2}\right)^{(l_4,k)}
	\left(\mathbf{E}_{\theta^{\star}}\left[\left|\epsilon_{0}^{\left(k\right)}\right|^4\right]-3\right)\\
	&\qquad+\frac{3}{2}\left(\Lambda_{\star}^{(l_1,l_3)}\Lambda_{\star}^{(l_2,l_4)}+\Lambda_{\star}^{(l_1,l_4)}\Lambda_{\star}^{(l_2,l_3)}\right),
	\end{align*}
	and with the function $\sigma$ as for $i=1,\ldots,d$ and $j=i,\ldots,d$,
	\begin{align*}
		\sigma\left(i,j\right):=\begin{cases}
		j&\text{ if }i=1,\\
		\sum_{\ell=1}^{i-1}\left(d-\ell+1\right)+j-i+1 & \text{ if }i>1,
		\end{cases}
	\end{align*}
	we define the matrix $W_{1}$ as for $i_{1},i_{2}=1,\ldots,d(d+1)/2$,
	\begin{align*}
		W_{1}^{\left(i_{1},i_{2}\right)}:=V\left(\sigma^{-1}\left(i_{1}\right),\sigma^{-1}\left(i_{2}\right)\right).
	\end{align*}
	\item Let
	\begin{align*}
	&\left\{B_{\kappa}(x)\left|\kappa=1,\ldots,m_1,\ B_{\kappa}=(B_{\kappa}^{(j_1,j_2)})_{j_1,j_2}\right.\right\},\\
	&\left\{f_{\lambda}(x)\left|\lambda=1,\ldots,m_2,\ f_{\lambda}=(f^{(1)}_{\lambda},\ldots,f^{(d)}_{\lambda})\right.\right\}
	\end{align*}
	be
	sequences of $\R^d\otimes \R^d$-valued functions and $\R^d$-valued ones respectively such that the components of themselves and their derivative with respect to $x$ are polynomial growth functions for all $\kappa$ and $\lambda$. 
	Then we define the following matrix-valued functionals, for $\bar{B}_{\kappa}:=\frac{1}{2}\left(B_{\kappa}+B_{\kappa}^T\right)$,
	\begin{align*}
	&\left(W_2^{(\tau)}\left(\left\{B_{\kappa}:\kappa=1,\ldots,m_{1}\right\}\right)\right)^{(\kappa_1,\kappa_2)}\\
	&\qquad:=\begin{cases}
		\nu\left(\mathrm{tr}\left\{\left(\bar{B}_{\kappa_1}A\bar{B}_{\kappa_2}A\right)(\cdot)\right\}\right) &\text{ if }\tau\in(1,2),\\
		\nu\left(\mathrm{tr}\left\{\left(\bar{B}_{\kappa_1}A\bar{B}_{\kappa_2}A+4\bar{B}_{\kappa_1}A\bar{B}_{\kappa_2}\Lambda_{\star}+12\bar{B}_{\kappa_1}\Lambda_{\star}\bar{B}_{\kappa_2}\Lambda_{\star}\right)(\cdot)\right\}\right)
		&\text{ if }\tau=2,
		\end{cases}\\
	&\left(W_3(\left\{f_{\lambda}:\lambda=1,\ldots,m_{2}\right\})\right)^{(\lambda_1,\lambda_2)}\\
	&\qquad:=
		\nu\left(\left(f_{\lambda_1}A\left(f_{\lambda_2}\right)^T\right)(\cdot)\right),
	\end{align*}
	where $\nu=\nu_{\theta^{\star}}$ is the invariant measure of $X_{t}$ discussed in the following assumption [A1]-(iv), and for all function $f$ on $\R^{d}$, $\nu\left(f\left(\cdot\right)\right):=\int_{\R^{d}} f\left(x\right)\nu\left(\mathrm{d}x\right)$.
\end{itemize}
With respect to $X_{t}$, we assume the following conditions.
\begin{itemize}
\item[{[A1]}]
\begin{itemize}
\item[(i)] $\inf_{x,\alpha}\det A\left(x,\alpha\right)>0$.
\item[(ii)] For some constant $C$, for all $x_{1},x_{2}\in\R^{d}$,
\begin{align*}
	\sup_{\alpha\in\Theta_{1}}\left\|a\left(x_{1},\alpha\right)-a\left(x_{2},\alpha\right)\right\|+
	\sup_{\beta\in\Theta_{2}}\left|b\left(x_{1},\beta\right)-b\left(x_{2},\beta\right)\right|\le C\left|x_{1}-x_{2}\right|
\end{align*}
\item[(iii)] For all $p\ge0$, $\sup_{t\ge0}\mathbf{E}_{\theta^{\star}}\left[\left|X_{t}\right|^{p}\right]<\infty$.
\item[(iv)] There exists an unique invariant measure $\nu=\nu_{0}$ on $\left(\R^{d},\mathcal{B}\left(\R^{d}\right)\right)$ and for all $p\ge1$ and $f\in L^{p}\left(\nu\right)$ with polynomial growth,
\begin{align*}
\frac{1}{T}\int_{0}^{T}f\left(X_{t}\right)\mathrm{d}t\to^{P}\int_{\R^{d}}f\left(x\right)\nu\left(\mathrm{d}x\right).
\end{align*}
\item[(v)] For any polynomial growth function $g:\R^{d}\to\R$  satisfying $\int_{R^{d}}g\left(x\right)\nu\left(\mathrm{d}x\right) = 0$, there exist $G(x)$, $\partial_{x^{\left(i\right)}}G(x)$
with at most polynomial growth for $i=1,\ldots,d$ such that for all $x\in\R^{d}$,
\begin{align*}
L_{\theta^{\star}}G\left(x\right)=-g\left(x\right),
\end{align*}
where $L_{\theta^{\star}}$ is the infinitesimal generator of $X_{t}$.
\end{itemize}
\end{itemize}

\begin{remark}
\cite{Paradoux-Veretennikov-2001} show a sufficient condition for [A1]-(v). \cite{Uchida-Yoshida-2012} also introduce the sufficient condition for [A1]-(iii)--(v) assuming [A1]-(i)--(ii), $\sup_{x,\alpha}A\left(x,\alpha\right)<\infty$ and $^\exists c_{0}>0$, $M_{0}>0$ and $\gamma\ge 0$ such that for all $\beta\in\Theta_{2}$ and $x\in\R^{d}$ satisfying $\left|x\right|\ge M_{0}$,
\begin{align*}
	\frac{1}{\left|x\right|}x^{T}b\left(x,\beta\right)&\le -c_{0}\left|x\right|^{\gamma}.
\end{align*}
\end{remark}

\begin{itemize}
\item[{[A2]}] There exists $C>0$ such that $a:\R^{d}\times \Theta_{1}\to \R^{d}\otimes \R^{r}$ and $b:\R^{d}\times \Theta_{2}\to \R^{d}$ have continuous derivatives satisfying
\begin{align*}
\sup_{\alpha\in\Theta_{1}}\left|\partial_{x}^{j}\partial_{\alpha}^{i}a\left(x,\alpha\right)\right|&\le
C\left(1+\left|x\right|\right)^{C},\ 0\le i\le 4,\ 0\le j\le 2,\\
\sup_{\beta\in\Theta_{2}}\left|\partial_{x}^{j}\partial_{\beta}^{i}b\left(x,\beta\right)\right|&\le C\left(1+\left|x\right|\right)^{C},\ 0\le i\le 4,\ 0\le j\le 2.
\end{align*}
\end{itemize}

With the invariant measure $\nu$, we define
{\small\begin{align*}
\mathbb{Y}_{1}^{\tau}\left(\alpha;\vartheta^{\star}\right)&:=-\frac{1}{2}\int \left\{\mathrm{tr}\left(A^{\tau}\left(x,\alpha,\Lambda_{\star}\right)^{-1}A^{\tau}\left(x,\alpha^{\star},\Lambda_{\star}\right)-I_{d}\right)+\log\frac{\det A^{\tau}\left(x,\alpha,\Lambda_{\star}\right)}{\det A^{\tau}\left(x,\alpha^{\star},\Lambda_{\star}\right)}\right\}\nu\left(\mathrm{d}x\right),\\
\mathbb{Y}_{2}\left(\beta;\vartheta^{\star}\right)&:=-\frac{1}{2}
\int A\left(x,\alpha^{\star}\right)^{-1}\left[\left(b\left(x,\beta\right)-b\left(x,\beta^{\star}\right)\right)^{\otimes 2}\right]\nu\left(\mathrm{d}x\right),
\end{align*}}
where $A^{\tau}\left(x,\alpha,\Lambda\right):=A\left(x,\alpha\right)+3\Lambda\mathbf{1}_{\left\{2\right\}}\left(\tau\right)$. For these functions, let us assume the following identifiability conditions hold.

\begin{itemize}
\item[{[A3]}] For all $\tau\in\left(1,2\right]$, there exists a constant $\chi\left(\alpha^{\star}\right)>0$ such that $\mathbb{Y}_{1}^{\tau}\left(\alpha;\theta^{\star}\right)\le -\chi\left(\theta^{\star}\right)\left|\alpha-\alpha^{\star}\right|^{2}$ for all $\alpha\in\Theta_{1}$.
\item[{[A4]}] For all $\tau\in\left(1,2\right]$, there exists a constant $\chi'\left(\beta^{\star}\right)>0$ such that $\mathbb{Y}_{2}\left(\beta;\theta^{\star}\right)\le -\chi'\left(\theta^{\star}\right)\left|\beta-\beta^{\star}\right|^{2}$ for all $\beta\in\Theta_{2}$.
\end{itemize}
The next assumption is with respect to the moments of noise.
\begin{itemize}
\item[{[A5]}] For any $k > 0$, $\varepsilon_{ih_{n}}$ has $k$-th moment and the components of $\varepsilon_{ih_{n}}$ are independent of the other
components for all $i$, $\left\{w_{t}\right\}_{t\ge0}$ and $x_{0}$. In addition, for all odd integer $k$, $i=0,\ldots,n$, $n\in\mathbf{N}$, and $\ell=1,\ldots,d$, $\mathbf{E}_{\theta^{\star}}\left[\left(\varepsilon_{ih_{n}}^{\left(\ell\right)}\right)^{k}\right]=0$, and $\mathbf{E}_{\theta^{\star}}\left[\varepsilon_{ih_{n}}^{\otimes 2}\right]=I_{d}$.
\end{itemize}

The assumption below determines the balance of convergence or divergence of several parameters. Note that $\tau$ is a tuning parameter and hence we can control it arbitrarily in its space $\left(1,2\right]$.
\begin{itemize}
\item[{[A6]}] $p_{n}=h_{n}^{-1/\tau}$, $\tau\in\left(1,2\right]$, $h_{n}\to0$, $T_{n}=nh_{n}\to\infty$, $k_{n}=n/p_{n}\to\infty$, $k_{n}\Delta_{n}^{2}\to0$ for $\Delta_{n}:=p_{n}h_{n}$. Furthermore, there exists $\epsilon_{0}>0$ such that $nh_{n}\ge k_{n}^{\epsilon_{0}}$ for sufficiently large $n$.
\end{itemize}

\begin{remark}\label{RemarkCentred}
Let us denote $\epsilon_{1}=\epsilon_{0}/2$ and $f\in \mathcal{C}^{1,1}\left(\R^{d}\times\Xi\right)$ where $f$ and the components of their derivatives are polynomial growth with respect to $x$ uniformly in $\vartheta\in\Xi$. Then the discussion in \cite{Uchida-2010} verifies under [A1] and [A6], for all $M>0$,
\begin{align*}
	\sup_{n\in\mathbf{N}}\mathbf{E}_{\theta^{\star}}\left[\sup_{\vartheta\in\Xi}\left(k_{n}^{\epsilon_{1}}\left|\frac{1}{k_{n}}\sum_{j=1}^{k_{n}-2}f\left(X_{j\Delta_{n}},\vartheta\right)-\int_{R^{d}}f\left(x,\vartheta\right)\nu\left(\mathrm{d}x\right)\right|\right)^{M}\right]<\infty.
\end{align*}
\end{remark}

\section{Quasi-likelihood analysis}

First of all, we introduce and analyse some quasi-likelihood functions and estimators which are defined in \cite{Nakakita-Uchida-2018a}. The quasi-likelihood functions for the diffusion parameter $\alpha$ and the drift one $\beta$ using this sequence are as follows:
{\begin{align*}
&\mathbb{H}_{1,n}^{\tau}\left(\alpha;\Lambda\right)
:=-\frac{1}{2}\sum_{j=1}^{k_{n}-2}
\left(\left(\frac{2}{3}\Delta_{n}A_{n}^{\tau}\left(\lm{Y}{j-1},\alpha,\Lambda\right)\right)^{-1}\left[\left(\lm{Y}{j+1}-\lm{Y}{j}\right)^{\otimes 2}\right]+\log\det A_{n}^{\tau}\left(\lm{Y}{j-1},\alpha,\Lambda\right)\right),\\
&\mathbb{H}_{2,n}\left(\beta;\alpha\right)
:=-\frac{1}{2}\sum_{j=1}^{k_{n}-2}
\left(\left(\Delta_{n}A\left(\lm{Y}{j-1},\alpha\right)\right)^{-1}\left[\left(\lm{Y}{j+1}-\lm{Y}{j}-\Delta_{n}b\left(\lm{Y}{j-1},\beta\right)\right)^{\otimes 2}\right]\right),
\end{align*}} where $A_{n}^{\tau}\left(x,\alpha,\Lambda\right):=A\left(x,\alpha\right)+3\Delta_{n}^{\frac{2-\tau}{\tau-1}}\Lambda$. 
We set the adaptive ML-type estimator $\hat{\Lambda}_{n}$, $\hat{\alpha}_{n}$ and $\hat{\beta}_{n}$ such that
\begin{align*}
\hat{\Lambda}_{n}&:=\frac{1}{2n}\sum_{i=0}^{n-1}\left(Y_{\left(i+1\right)h_{n}}-Y_{ih_{n}}\right)^{\otimes 2},\\
\mathbb{H}_{1,n}^{\tau}\left(\hat{\alpha}_{n};\hat{\Lambda}_{n}\right)
&=\sup_{\alpha\in\Theta_{1}}\mathbb{H}_{1,n}^{\tau}\left(\alpha;\hat{\Lambda}_{n}\right),\\
\mathbb{H}_{2,n}\left(\hat{\beta}_{n};\hat{\alpha}_{n}\right)
&=\sup_{\beta\in\Theta_{2}}\mathbb{H}_{2,n}\left(\beta;\hat{\alpha}_{n}\right).
\end{align*}
Assume that $\pi_{\ell}$, $\ell=1,2$ are continuous and $0<\inf_{\theta_{\ell}\in\Theta_{\ell}}\pi_{\ell}\left(\theta_{\ell}\right)<\sup_{\theta_{\ell}\in\Theta_{\ell}}\pi_{\ell}\left(\theta_{\ell}\right)<\infty$, and denote the adaptive Bayes-type estimators
\begin{align*}
\tilde{\alpha}_{n}&:=\left\{\int_{\Theta_{1}}\exp\left(\mathbb{H}_{1,n}^{\tau}\left(\alpha;\hat{\Lambda}_{n}\right)\right)\pi_{1}\left(\alpha\right)\mathrm{d}\alpha\right\}^{-1}\int_{\Theta_{1}}\alpha\exp\left(\mathbb{H}_{1,n}^{\tau}\left(\alpha;\hat{\Lambda}_{n}\right)\right)\pi_{1}\left(\alpha\right)\mathrm{d}\alpha,\\
\tilde{\beta}_{n}&:=\left\{\int_{\Theta_{2}}\exp\left(\mathbb{H}_{2,n}\left(\beta;\tilde{\alpha}_{n}\right)\right)\pi_{2}\left(\beta\right)\mathrm{d}\beta\right\}^{-1}\int_{\Theta_{2}}\beta\exp\left(\mathbb{H}_{2,n}\left(\beta;\tilde{\alpha}_{n}\right)\right)\pi_{2}\left(\beta\right)\mathrm{d}\beta.
\end{align*}

Our purpose is to show the polynomial-type large deviation inequalities for the quasi-likelihood functions defined above in the framework introduced by \cite{Yoshida-2011}, and the convergences of moments for these estimators as the application of them.
Let us denote the following statistical random fields for $u_{1}\in\R^{m_{1}}$ and $u_{2}\in\R^{m_{2}}$
\begin{align*}
	\mathbb{Z}_{1,n}^{\tau}\left(u_{1};\hat{\Lambda}_{n},\alpha^{\star}\right)&:=\exp\left(\mathbb{H}_{1,n}^{\tau}\left(\alpha^{\star}+k_{n}^{-1/2}u_{1};\hat{\Lambda}_{n}\right)-\mathbb{H}_{1,n}^{\tau}\left(\alpha^{\star};\hat{\Lambda}_{n}\right)\right),\\
	\mathbb{Z}_{2,n}^{\mathrm{ML}}\left(u_{2};\hat{\alpha}_{n},\beta^{\star}\right)&:=\exp\left(\mathbb{H}_{2,n}\left(\beta^{\star}+T_{n}^{-1/2}u_{2};\hat{\alpha}_{n}\right)-\mathbb{H}_{2,n}\left(\beta^{\star};\hat{\alpha}_{n}\right)\right),\\
	\mathbb{Z}_{2,n}^{\mathrm{Bayes}}\left(u_{2};\tilde{\alpha}_{n},\beta^{\star}\right)&:=\exp\left(\mathbb{H}_{2,n}\left(\beta^{\star}+T_{n}^{-1/2}u_{2};\tilde{\alpha}_{n}\right)-\mathbb{H}_{2,n}\left(\beta^{\star};\tilde{\alpha}_{n}\right)\right),
\end{align*}
and some sets
\begin{align*}
\mathbb{U}_{1,n}^{\tau}\left(\alpha^{\star}\right)&:=\left\{u_{1}\in\R^{m_{1}};\alpha^{\star}+k_{n}^{-1/2}u_{1}\in\Theta_{1}\right\},\\
\mathbb{U}_{2,n}\left(\beta^{\star}\right)&:=\left\{u_{2}\in\R^{m_{2}};\beta^{\star}+T_{n}^{-1/2}u_{2}\in\Theta_{2}\right\},
\end{align*}
and for $r\ge 0$,
\begin{align*}
V_{1,n}^{\tau}\left(r,\alpha^{\star}\right)&:=\left\{u_{1}\in\mathbb{U}_{1,n}^{\tau}\left(\alpha^{\star}\right);r\le \left|u_{1}\right|\right\},\\
V_{2,n}\left(r,\beta^{\star}\right)&:=\left\{u_{2}\in\mathbb{U}_{2,n}\left(\beta^{\star}\right);r\le \left|u_{2}\right|\right\}.
\end{align*}

We use the notation as \cite{Nakakita-Uchida-2018a} for the information matrices
\begin{align*}
\mathcal{I}^{\tau}\left(\vartheta^{\star}\right)&:=\mathrm{diag}\left\{W_{1}, \mathcal{I}^{(2,2),\tau},\mathcal{I}^{(3,3)}\right\}\left(\vartheta^{\star}\right),\\
\mathcal{J}^{\tau}\left(\vartheta^{\star}\right)&:=\mathrm{diag}\left\{I_{d(d+1)/2},\mathcal{J}^{(2,2),\tau},\mathcal{J}^{(3,3)}\right\}(\vartheta^{\star}),
\end{align*}
where for $i_1,i_2\in\left\{1,\ldots,m_1\right\}$,
\begin{align*}
\mathcal{I}^{(2,2),\tau}(\vartheta^{\star})&:=
W_2^{(\tau)}\left(\left\{\frac{3}{4}\left( A^{\tau}\right)^{-1}\left(\partial_{\alpha^{(k_1)}}A\right)\left(A^{\tau}\right)^{-1}(\cdot,\vartheta^{\star}):k_{1}=1,\ldots,m_{1}\right\}\right),\\
\mathcal{J}^{(2,2),\tau}(\vartheta^{\star})&:=
\left[\frac{1}{2}\nu\left(\mathrm{tr}\left\{\left(A^{\tau}\right)^{-1}\left(\partial_{\alpha^{(i_1)}}A\right)\left(A^{\tau}\right)^{-1}
	\left(\partial_{\alpha^{(i_2)}}A\right)\right\}(\cdot,\vartheta^{\star})\right)\right]_{i_1,i_2},
\end{align*}
and for $j_1,j_2\in\left\{1,\ldots,m_2\right\}$,
\begin{align*}
\mathcal{I}^{(3,3)}(\theta^{\star})&=\mathcal{J}^{(3,3)}(\theta^{\star}):=
\left[\nu\left(\left(A\right)^{-1}\left[\partial_{\beta^{(j_1)}}b,\partial_{\beta^{(j_2)}}b\right](\cdot,\theta^{\star})\right)\right]_{j_1,j_2}.
\end{align*}
We also denote $\hat{\theta}_{\varepsilon,n}:=\mathrm{vech}\hat{\Lambda}_{n}$ and $\theta_{\varepsilon}^{\star}:=\mathrm{vech}\Lambda_{\star}$.

\begin{theorem}\label{mainthm}
Under [A1]-[A6], we have the following results.
\begin{enumerate}
	\item The polynomial-type large deviation inequalities hold: for all $L>0$, there exists a constant $C\left(L\right)$ such that for all $r>0$,
	\begin{align*}
	P_{\theta^{\star}}\left[\sup_{u_{1}\in V_{1,n}^{\tau}\left(r,\alpha^{\star}\right)}\mathbb{Z}_{1,n}^{\tau}\left(u_{1};\hat{\Lambda}_{n},\alpha^{\star}\right)\ge e^{-r}\right] &\le \frac{C\left(L\right)}{r^{L}},\\
	P_{\theta^{\star}}\left[\sup_{u_{2}\in V_{2,n}\left(r,\beta^{\star}\right)}\mathbb{Z}_{2,n}^{\mathrm{ML}}\left(u_{2};\hat{\alpha}_{n},\beta^{\star}\right)\ge e^{-r}\right] &\le \frac{C\left(L\right)}{r^{L}},\\
	P_{\theta^{\star}}\left[\sup_{u_{2}\in V_{2,n}\left(r,\beta^{\star}\right)}\mathbb{Z}_{2,n}^{\mathrm{Bayes}}\left(u_{2};\tilde{\alpha}_{n},\beta^{\star}\right)\ge e^{-r}\right] &\le \frac{C\left(L\right)}{r^{L}}.
	\end{align*}
	\item The convergences of moment hold:
	\begin{align*}
		\mathbf{E}_{\theta^{\star}}\left[f\left(\sqrt{n}\left(\hat{\theta}_{\varepsilon,n}-\theta_{\varepsilon}^{\star}\right),
		\sqrt{k_{n}}\left(\hat{\alpha}_{n}-\alpha^{\star}\right),
		\sqrt{T_{n}}\left(\hat{\beta}_{n}-\beta^{\star}\right)\right)\right]&\to \mathbb{E}\left[f\left(\zeta_{0},\zeta_{1},\zeta_{2}\right)\right],\\
		\mathbf{E}_{\theta^{\star}}\left[f\left(\sqrt{n}\left(\hat{\theta}_{\varepsilon,n}-\theta_{\varepsilon}^{\star}\right),\sqrt{k_{n}}\left(\tilde{\alpha}_{n}-\alpha^{\star}\right),
		\sqrt{T_{n}}\left(\tilde{\beta}_{n}-\beta^{\star}\right)\right)\right]&\to \mathbb{E}\left[f\left(\zeta_{0},\zeta_{1},\zeta_{2}\right)\right],
	\end{align*}
	where
	\begin{align*}
		\left(\zeta_{0},\zeta_{1},\zeta_{2}\right)
		\sim N_{d\left(d+1\right)/2+m_{1}+m_{2}}\left(\mathbf{0},\left(\mathcal{J}^{\tau}\left(\vartheta^{\star}\right)\right)^{-1}\left(\mathcal{I}^{\tau}\left(\vartheta^{\star}\right)\right)\left(\mathcal{J}^{\tau}\left(\vartheta^{\star}\right)\right)^{-1}
		\right)
	\end{align*}
	and $f$ is an arbitrary continuous functions of at most polynomial growth.
\end{enumerate}
\end{theorem}

\subsection{Evaluation for local means}

In the first place we give some evaluations related to local means. Some of the instruments are inherited from the previous studies by \cite{Nakakita-Uchida-2017} and \cite{Nakakita-Uchida-2018a}.
We define the following random variables:
\begin{align*}
\zeta_{j+1,n}:=
\frac{1}{p_{n}}\sum_{i=0}^{p_{n}-1}\int_{j\Delta_{n}+ih_{n}}^{\left(j+1\right)\Delta_{n}}\mathrm{d}w_{s},\quad
\zeta_{j+2,n}':=
\frac{1}{p_{n}}\sum_{i=0}^{p_{n}-1}\int_{\left(j+1\right)\Delta_{n}}^{\left(j+1\right)\Delta_{n}+ih_{n}}\mathrm{d}w_{s}.
\end{align*}

The next lemma is Lemma 11 in \cite{Nakakita-Uchida-2018a}.

\begin{lemma}\label{EvalZeta}
	$\zeta_{j+1,n}$ and $\zeta_{j+1,n}'$ are $\mathcal{G}_{j+1}^{n}$-measurable, independent of $\mathcal{G}_{j}^{n}$ and Gaussian.These variables have the next decompositions:
	\begin{align*}
	\zeta_{j+1,n}&=\frac{1}{p_{n}}\sum_{k=0}^{p_{n}-1}\left(k+1\right)\int_{j\Delta_{n}+kh_{n}}^{j\Delta_{n}+\left(k+1\right)h_{n}}\mathrm{d}w_{s},\\
	\zeta_{j+1,n}'&=\frac{1}{p_{n}}\sum_{k=0}^{p_{n}-1}\left(p_{n}-k-1\right)
	\int_{j\Delta_{n}+kh_{n}}^{j\Delta_{n}+\left(k+1\right)h_{n}}\mathrm{d}w_{s}.
	\end{align*}
	The evaluation of the following conditional expectations holds:
	\begin{align*}
	\mathbf{E}_{\theta^{\star}}\left[\zeta_{j,n}|\mathcal{G}_{j}^{n}\right]=\mathbf{E}_{\theta^{\star}}\left[\zeta_{j+1,n}'|\mathcal{G}_{j}^{n}\right]&=\mathbf{0},\\
	\mathbf{E}_{\theta^{\star}}\left[\zeta_{j+1,n}\left(\zeta_{j+1,n}\right)^{T}|\mathcal{G}_{j}^{n}\right]
	&=m_{n}\Delta_{n}I_{r},\\
	\mathbf{E}_{\theta^{\star}}\left[\zeta_{j+1,n}'\left(\zeta_{j+1,n}'\right)^{T}|\mathcal{G}_{j}^{n}\right]&=m_{n}'\Delta_{n}I_{r},\\
	\mathbf{E}_{\theta^{\star}}\left[\zeta_{j+1,n}\left(\zeta_{j+1,n}'\right)^{T}|\mathcal{G}_{j}^{n}\right]&=\chi_{n}\Delta_{n}I_{r},
	\end{align*}
	where $m_{n}=\left(\frac{1}{3}+\frac{1}{2p_{n}}+\frac{1}{6p_{n}^2}\right)$, $m_{n}'=\left(\frac{1}{3}-\frac{1}{2p_{n}}+\frac{1}{6p_{n}^2}\right)$, and $\chi_{n}=\frac{1}{6}\left(1-\frac{1}{p_{n}^2}\right)$.
\end{lemma}

The next lemma can be obtained with same discussion as Proposition 12 in \cite{Nakakita-Uchida-2018a}.

\begin{lemma}\label{ApproxLM}
	Assume the component of the function $f\in C^{1}\left(\R^{d}\times \Xi;\ \R\right)$ and $\partial_{x}f$ are polynomial growth functions uniformly in $\vartheta\in\Xi$. 
	For all $p\ge 1$, there exists $C\left(p\right)>0$ such that for all $n\in\mathbf{N}$,
	\begin{align*}
	\sup_{j=0,\ldots,k_{n}-1}
	\left\|\sup_{\vartheta\in\Xi}
	\left|f\left(\lm{Y}{j},\vartheta\right)-f\left(X_{j\Delta_{n}},\vartheta\right)\right|
	\right\|_{p} \le C\left(p\right)\Delta_{n}^{1/2}.
	\end{align*}
\end{lemma}

\begin{lemma}\label{ApproxSumLM}
	Assume the component of the function $f\in C^{1}\left(\R^{d}\times \Xi;\ \R\right)$ and $\partial_{x}f$ are polynomial growth functions uniformly in $\vartheta\in\Xi$. For all $p\ge 1$, there exists $C\left(p\right)>0$ such that for all $n\in\mathbf{N}$
	\begin{align*}
	\left\|\sup_{\vartheta\in\Xi}
	\left|\frac{1}{k_{n}}\sum_{j=1}^{k_{n}-2}f\left(\lm{Y}{j},\vartheta\right)-\frac{1}{k_{n}}\sum_{j=1}^{k_{n}-2}f\left(X_{j\Delta_{n}},\vartheta\right)\right|
	\right\|_{p} \le C\left(p\right)\Delta_{n}^{1/2}.
	\end{align*}
\end{lemma}

\begin{proof}
	By Lemma \ref{ApproxLM},
	\begin{align*}
	\left\|\sup_{\vartheta\in\Xi}
	\left|\frac{1}{k_{n}}\sum_{j=1}^{k_{n}-2}f\left(\lm{Y}{j},\vartheta\right)-\frac{1}{k_{n}}\sum_{j=1}^{k_{n}-2}f\left(X_{j\Delta_{n}},\vartheta\right)\right|
	\right\|_{p} &\le \left\|\frac{1}{k_{n}}\sum_{j=1}^{k_{n}-2}\sup_{\vartheta\in\Xi}
	\left|f\left(\lm{Y}{j},\vartheta\right)-f\left(X_{j\Delta_{n}},\vartheta\right)\right|
	\right\|_{p}\\
	&\le \frac{1}{k_{n}}\sum_{j=1}^{k_{n}-2}\left\|\sup_{\vartheta\in\Xi}
	\left|f\left(\lm{Y}{j},\vartheta\right)-f\left(X_{j\Delta_{n}},\vartheta\right)\right|
	\right\|_{p}\\
	&\le C\left(p\right)\Delta_{n}^{1/2}.
	\end{align*}
\end{proof}

\begin{lemma}\label{ApproxFunctional}
	Assume the components of the functions $f,g\in C^{2}\left(\R^{d};\ \R\right)$, $\partial_{x}f$, $\partial_{x}g$, $\partial_{x}^{2}f$ $\partial_{x}^{2}g$ are polynomial growth functions. Then we have
	\begin{align*}
	\left|\mathbf{E}\left[f\left(\lm{Y}{j}\right)g\left(X_{\left(j+1\right)\Delta_{n}}\right)-f\left(X_{j\Delta_{n}}\right)g\left(X_{j\Delta_{n}}\right)|\mathcal{H}_{j}^{n}\right]\right|\le C\Delta_{n}\left(1+\left|X_{j\Delta_{n}}\right|\right)^{C}.
	\end{align*}
\end{lemma}

\begin{proof}
	For Taylor's expansion, we have
	\begin{align*}
	&f\left(\lm{Y}{j}\right)g\left(X_{\left(j+1\right)\Delta_{n}}\right)\\
	&=f\left(X_{j\Delta_{n}}\right)g\left(X_{\left(j+1\right)\Delta_{n}}\right)+\partial_{x}f\left(X_{j\Delta_{n}}\right)\left[\lm{Y}{j}-X_{j\Delta_{n}}\right]g\left(X_{\left(j+1\right)\Delta_{n}}\right)\\
	&\quad+\int_{0}^{1}\left(1-s\right)\partial_{x}^{2}f\left(X_{j\Delta_{n}}+s\left(\lm{Y}{j}-X_{j\Delta_{n}}\right)\mathrm{d}s\right)\left[\left(\lm{Y}{j}-X_{j\Delta_{n}}\right)^{\otimes2}\right]g\left(X_{\left(j+1\right)\Delta_{n}}\right),
	\end{align*}
	and Ito-Taylor expansion and Proposition 3.2 in \citet{Favetto-2014} verify
	\begin{align*}
	&\left|\mathbf{E}_{\theta^{\star}}\left[f\left(X_{j\Delta_{n}}\right)g\left(X_{\left(j+1\right)\Delta_{n}}\right)-f\left(X_{j\Delta_{n}}\right)g\left(X_{j\Delta_{n}}\right)|\mathcal{H}_{j}^{n}\right]\right|\le C\Delta_{n}\left(1+\left|X_{j\Delta_{n}}\right|\right)^{C},\\
	&	\left|\mathbf{E}_{\theta^{\star}}\left[\int_{0}^{1}\left(1-s\right)\partial_{x}^{2}f\left(X_{j\Delta_{n}}+s\left(\lm{Y}{j}-X_{j\Delta_{n}}\right)\mathrm{d}s\right)\left[\left(\lm{Y}{j}-X_{j\Delta_{n}}\right)^{\otimes2}\right]g\left(X_{\left(j+1\right)\Delta_{n}}\right)|\mathcal{H}_{j}^{n}\right]\right|\\
	&\qquad\le \mathbf{E}_{\theta^{\star}}\left[C\left(1+\sup_{t\in\left[j\Delta_{n},\left(j+1\right)\Delta_{n}\right]}\left|X_{t}\right|+\left|\lm{\varepsilon}{j}\right|\right)^{C}|\mathcal{H}_{j}^{n}\right]
	\mathbf{E}_{\theta^{\star}}\left[\left|\lm{Y}{j}-X_{j\Delta_{n}}\right|^{2}|\mathcal{H}_{j}^{n}\right]\\
	&\qquad\le C\Delta_{n}\left(1+\left|X_{j\Delta_{n}}\right|\right)^{C}.
	\end{align*}
	It holds that
	\begin{align*}
	&\partial_{x}f\left(X_{j\Delta_{n}}\right)\left[\lm{Y}{j}-X_{j\Delta_{n}}\right]g\left(X_{\left(j+1\right)\Delta_{n}}\right)\\
	&=\partial_{x}f\left(X_{j\Delta_{n}}\right)\left[\lm{Y}{j}-X_{j\Delta_{n}}\right]g\left(X_{j\Delta_{n}}\right)
	\\
	&\quad+\partial_{x}f\left(X_{j\Delta_{n}}\right)\left[\lm{Y}{j}-X_{j\Delta_{n}}\right]\int_{0}^{1}\partial_{x}g\left(X_{j\Delta_{n}}+s\left(X_{\left(j+1\right)\Delta_{n}}-X_{j\Delta_{n}}\right)\right)\mathrm{d}s\left[X_{\left(j+1\right)\Delta_{n}}-X_{j\Delta_{n}}\right]
	\end{align*}
	and Proposition 3.2 in \citet{Favetto-2014} leads to
	\begin{align*}
	&\left|\mathbf{E}_{\theta^{\star}}\left[\partial_{x}f\left(X_{j\Delta_{n}}\right)\left[\lm{Y}{j}-X_{j\Delta_{n}}\right]g\left(X_{j\Delta_{n}}\right)|\mathcal{H}_{j}^{n}\right]\right|\le C\Delta_{n}\left(1+\left|X_{j\Delta_{n}}\right|\right)^{C},\\
	&\left|\mathbf{E}_{\theta^{\star}}\left[\partial_{x}f\left(X_{j\Delta_{n}}\right)\left[\lm{Y}{j}-X_{j\Delta_{n}}\right]\right.\right.\\
	&\qquad\times\left.\left.\int_{0}^{1}\partial_{x}g\left(X_{j\Delta_{n}}+s\left(X_{\left(j+1\right)\Delta_{n}}-X_{j\Delta_{n}}\right)\right)\mathrm{d}s\left[X_{\left(j+1\right)\Delta_{n}}-X_{j\Delta_{n}}\right]|\mathcal{H}_{j}^{n}\right]\right|\\
	&\quad\le C\left(1+\left|X_{j\Delta_{n}}\right|\right)^{C}
	\mathbf{E}_{\theta^{\star}}\left[\left|\lm{Y}{j}-X_{j\Delta_{n}}\right|^{2}|\mathcal{H}_{j}^{n}\right]^{1/2}	\mathbf{E}_{\theta^{\star}}\left[\left|X_{\left(j+1\right)\Delta_{n}}-X_{j\Delta_{n}}\right|^{2}|\mathcal{H}_{j}^{n}\right]^{1/2}\\
	&\quad\le C\Delta_{n}\left(1+\left|X_{j\Delta_{n}}\right|\right)^{C}.
	\end{align*}
	Hence we obtain the result.
\end{proof}

\begin{lemma}\label{ExpansionLM}
	\begin{enumerate}
		\item[{(i)}] The next expansion holds:
		\begin{align*}
		\lm{Y}{j+1}-\lm{Y}{j}= \Delta_{n}b\left(X_{j\Delta_{n}}\right)+ a\left(X_{j\Delta_{n}}\right)\left(\zeta_{j+1,n}+\zeta_{j+2,n}'\right)+e_{j,n}
		+\left(\Lambda_{\star}\right)^{1/2}\left(\lm{\varepsilon}{j+1}-\lm{\varepsilon}{j}\right)
		\end{align*}
		where $e_{j,n}$ is a $\mathcal{H}_{j+2}^{n}$-measurable random variable such that $\left\|e_{j,n}\right\|_{p}\le C\left(p\right)\Delta_{n}$, for $j=1,\ldots,k_{n}-2$, $n\in\mathbf{N}$ and $p\ge 1$.
		\item[(ii)] For any $p\ge 1$ and $\mathcal{H}_{j}^{n}$-measurable $\R^{d}\otimes \R^{r}$-valued random variable $\mathbb{B}_{j}^{n}$ such that $\sup_{j}\mathbf{E}\left[\left\|\mathbb{B}_{j}^{n}\right\|^{m}\right]<\infty$ for all $m\in\mathbf{N}$, we have the next $L^{p}$-boundedness:
		\begin{align*}
		\mathbf{E}_{\theta^{\star}}\left[\left|\sum_{j=1}^{k_{n}-2}\mathbb{B}_{j}^{n}\left[e_{j,n}\left(\zeta_{j+1,n}+\zeta_{j+2,n}'\right)^{T}\right]\right|^{p}\right]^{1/p}\le C\left(p\right)k_{n}\Delta_{n}^{2}.
		\end{align*}
		\item[(iii)] For any $p\ge 1$ and $\mathcal{H}_{j}^{n}$-measurable $\R^{d}$-valued random variable $\mathbb{C}_{j}^{n}$ such that $\sup_{j}\mathbf{E}\left[\left|\mathbb{C}_{j}^{n}\right|^{m}\right]<\infty$ for all $m\in\mathbf{N}$, we have the next $L^{p}$-boundedness:
		\begin{align*}
		\mathbf{E}_{\theta^{\star}}\left[\left|\sum_{j=1}^{k_{n}-2}\mathbb{C}_{j}^{n}\left[e_{j,n}\right]\right|^{p}\right]^{1/p}\le C\left(p\right)k_{n}\Delta_{n}^{3/2}.
		\end{align*}
	\end{enumerate}
\end{lemma}

\begin{proof}
	Firstly we prove (i). Without loss of generality, assume $p$ is an even number. It holds
	\begin{align*}
	\lm{Y}{j+1}-\lm{Y}{j}=\lm{X}{j+1}-\lm{X}{j}+\left(\Lambda_{\star}\right)^{1/2}\left(\lm{\varepsilon}{j+1}-\lm{\varepsilon}{j}\right),
	\end{align*}
	and
	\begin{align*}
	&\lm{X}{j+1}-\lm{X}{j}\\
	&=\frac{1}{p_{n}}\sum_{i=0}^{p_{n}-1}\left(X_{\left(j+1\right)\Delta_{n}+ih_{n}}-X_{j\Delta_{n}+ih_{n}}\right)\\
	&=\frac{1}{p_{n}}\sum_{i=0}^{p_{n}-1}\left(X_{\left(j+1\right)\Delta_{n}+ih_{n}}-X_{\left(j+1\right)\Delta_{n}+\left(i-1\right)h_{n}}+X_{\left(j+1\right)\Delta_{n}+\left(i-1\right)h_{n}}-\cdots-X_{j\Delta_{n}+ih_{n}}\right)\\
	&=\frac{1}{p_{n}}\sum_{i=0}^{p_{n}-1}\left(\int_{j\Delta_{n}+\left(p_{n}+i-1\right)h_{n}}^{j\Delta_{n}+\left(p_{n}+i\right)h_{n}}\mathrm{d}X_{s}+\cdots +\int_{j\Delta_{n}+ih_{n}}^{j\Delta_{n}+\left(i+1\right)h_{n}}\mathrm{d}X_{s}\right)\\
	&=\frac{1}{p_{n}}\sum_{i=0}^{p_{n}-1}
	\sum_{l=0}^{p_{n}-1}\int_{j\Delta_{n}+\left(i+l\right)h_{n}}^{j\Delta_{n}+\left(i+l+1\right)h_{n}}\mathrm{d}X_{s}\\
	&=\frac{1}{p_{n}}\sum_{i=0}^{p_{n}-1}
	\left(i+1\right)\int_{j\Delta_{n}+ih_{n}}^{j\Delta_{n}+\left(i+1\right)h_{n}}\mathrm{d}X_{s}
	+\frac{1}{p_{n}}\sum_{i=0}^{p_{n}-1}
	\left(p_{n}-i-1\right)\int_{\left(j+1\right)\Delta_{n}+ih_{n}}^{\left(j+1\right)\Delta_{n}+\left(i+1\right)h_{n}}\mathrm{d}X_{s}\\
	&=\frac{1}{p_{n}}\sum_{i=0}^{p_{n}-1}
	\left(i+1\right)\int_{j\Delta_{n}+ih_{n}}^{j\Delta_{n}+\left(i+1\right)h_{n}}\mathrm{d}X_{s}
	+\frac{1}{p_{n}}\sum_{i=0}^{p_{n}-1}
	\left(p_{n}-i-1\right)\int_{\left(j+1\right)\Delta_{n}+ih_{n}}^{\left(j+1\right)\Delta_{n}+\left(i+1\right)h_{n}}\mathrm{d}X_{s}\\
	&\qquad+\Delta_{n}b\left(X_{j\Delta_{n}}\right)
	-\frac{1}{p_{n}}\sum_{i=0}^{p_{n}-1}\left(\left(i+1\right)h_{n}+\left(p_{n}-i-1\right)h_{n}\right)b\left(X_{j\Delta_{n}}\right)\\
	&=\Delta_{n}b\left(X_{j\Delta_{n}}\right)+ a\left(X_{j\Delta_{n}}\right)\left(\zeta_{j+1,n}+\zeta_{j+2,n}'\right)+e_{j,n},
	\end{align*}
	where $e_{j,n}=\sum_{l=1}^{3}\left(r_{j,n}^{\left(l\right)}+s_{j,n}^{\left(l\right)}\right)$,
	\begin{align*}
	r_{j,n}^{\left(1\right)}&=\frac{1}{p_{n}}\sum_{i=0}^{p_{n}-1}\left(i+1\right)\int_{j\Delta_{n}+ih_{n}}^{j\Delta_{n}+\left(i+1\right)h_{n}}\left(a\left(X_{j\Delta_{n}+ih_{n}}\right)-a\left(X_{j\Delta_{n}}\right)\right)\mathrm{d}w_{s},\\
	r_{j,n}^{\left(2\right)}&=\frac{1}{p_{n}}\sum_{i=0}^{p_{n}-1}\left(i+1\right)\int_{j\Delta_{n}+ih_{n}}^{j\Delta_{n}+\left(i+1\right)h_{n}}\left(a\left(X_{s}\right)-a\left(X_{j\Delta_{n}+ih_{n}}\right)\right)\mathrm{d}w_{s},\\
	r_{j,n}^{\left(3\right)}&=\frac{1}{p_{n}}\sum_{i=0}^{p_{n}-1}\left(i+1\right)\int_{j\Delta_{n}+ih_{n}}^{j\Delta_{n}+\left(i+1\right)h_{n}}\left(b\left(X_{s}\right)-b\left(X_{j\Delta_{n}}\right)\right)\mathrm{d}s,\\
	s_{j,n}^{\left(1\right)}&=\frac{1}{p_{n}}\sum_{i=0}^{p_{n}-1}\left(p_{n}-i-1\right)\int_{\left(j+1\right)\Delta_{n}+ih_{n}}^{\left(j+1\right)\Delta_{n}+\left(i+1\right)h_{n}}\left(a\left(X_{\left(j+1\right)\Delta_{n}+ih_{n}}\right)-a\left(X_{j\Delta_{n}}\right)\right)\mathrm{d}w_{s},\\
	s_{j,n}^{\left(2\right)}&=\frac{1}{p_{n}}\sum_{i=0}^{p_{n}-1}\left(p_{n}-i-1\right)\int_{\left(j+1\right)\Delta_{n}+ih_{n}}^{\left(j+1\right)\Delta_{n}+\left(i+1\right)h_{n}}\left(a\left(X_{s}\right)-a\left(X_{\left(j+1\right)\Delta_{n}+ih_{n}}\right)\right)\mathrm{d}w_{s},\\
	s_{j,n}^{\left(3\right)}&=\frac{1}{p_{n}}\sum_{i=0}^{p_{n}-1}\left(p_{n}-i-1\right)\int_{\left(j+1\right)\Delta_{n}+ih_{n}}^{\left(j+1\right)\Delta_{n}+\left(i+1\right)h_{n}}\left(b\left(X_{s}\right)-b\left(X_{j\Delta_{n}}\right)\right)\mathrm{d}s,
	\end{align*}
	using Lemma \ref{EvalZeta}. 
	By BDG inequality, H\"{o}lder's inequality, and triangular inequality for $L^{p/2}$-norm, we have
	\begin{align*}
	\left\|r_{j,n}^{\left(1\right)}\right\|_{p}
	&=\mathbf{E}_{\theta^{\star}}\left[\left|\frac{1}{p_{n}}\sum_{i=0}^{p_{n}-1}\left(i+1\right)\int_{j\Delta_{n}+ih_{n}}^{j\Delta_{n}+\left(i+1\right)h_{n}}\left(a\left(X_{j\Delta_{n}+ih_{n}}\right)-a\left(X_{j\Delta_{n}}\right)\right)\mathrm{d}w_{s}\right|^{p}\right]^{1/p}\\
	&\le C\left(p\right)\mathbf{E}_{\theta^{\star}}\left[\left|\frac{1}{p_{n}^{2}}\sum_{i=0}^{p_{n}-1}\left(i+1\right)^{2}\int_{j\Delta_{n}+ih_{n}}^{j\Delta_{n}+\left(i+1\right)h_{n}}\left\|a\left(X_{j\Delta_{n}+ih_{n}}\right)-a\left(X_{j\Delta_{n}}\right)\right\|^{2}\mathrm{d}s\right|^{p/2}\right]^{1/p}\\
	&\le C\left(p\right)\mathbf{E}_{\theta^{\star}}\left[\left|\sum_{i=0}^{p_{n}-1}\int_{j\Delta_{n}+ih_{n}}^{j\Delta_{n}+\left(i+1\right)h_{n}}\left\|a\left(X_{j\Delta_{n}+ih_{n}}\right)-a\left(X_{j\Delta_{n}}\right)\right\|^{2}\mathrm{d}s\right|^{p/2}\right]^{1/p}\\
	&= C\left(p\right)\left(\mathbf{E}_{\theta^{\star}}\left[\left|\sum_{i=0}^{p_{n}-1}\int_{j\Delta_{n}+ih_{n}}^{j\Delta_{n}+\left(i+1\right)h_{n}}\left\|a\left(X_{j\Delta_{n}+ih_{n}}\right)-a\left(X_{j\Delta_{n}}\right)\right\|^{2}\mathrm{d}s\right|^{p/2}\right]^{2/p}\right)^{1/2}\\
	&\le  C\left(p\right)\left(\sum_{i=0}^{p_{n}-1}\mathbf{E}_{\theta^{\star}}\left[\left|\int_{j\Delta_{n}+ih_{n}}^{j\Delta_{n}+\left(i+1\right)h_{n}}\left\|a\left(X_{j\Delta_{n}+ih_{n}}\right)-a\left(X_{j\Delta_{n}}\right)\right\|^{2}\mathrm{d}s\right|^{p/2}\right]^{2/p}\right)^{1/2}\\
	&\le C\left(p\right)\left(\sum_{i=0}^{p_{n}-1}h_{n}^{1-2/p}\mathbf{E}_{\theta^{\star}}\left[\int_{j\Delta_{n}+ih_{n}}^{j\Delta_{n}+\left(i+1\right)h_{n}}\left\|a\left(X_{j\Delta_{n}+ih_{n}}\right)-a\left(X_{j\Delta_{n}}\right)\right\|^{p}\mathrm{d}s\right]^{2/p}\right)^{1/2}\\
	&= C\left(p\right)\left(\sum_{i=0}^{p_{n}-1}h_{n}\sup_{s\in\left[j\Delta_{n},\left(j+1\right)\Delta_{n}\right]}\mathbf{E}_{\theta^{\star}}\left[\left\|a\left(X_{s}\right)-a\left(X_{j\Delta_{n}}\right)\right\|^{p}\right]^{2/p}\right)^{1/2}\\
	&\le C\left(p\right)\left(\sum_{i=0}^{p_{n}-1}h_{n}\left(C\left(p\right)\Delta_{n}^{p/2}\mathbf{E}_{\theta^{\star}}\left[\left(1+\left|X_{j\Delta_{n}}\right|\right)^{C\left(p\right)}\right]\right)^{2/p}\right)^{1/2}\\
	&\le C\left(p\right)\left(C\left(p\right)\Delta_{n}^{2}\right)^{1/2}\\
	&\le C\left(p\right)\Delta_{n}
	\end{align*}
	and we also have $\left\|s_{j,n}^{\left(1\right)}\right\|_{p}\le C\left(p\right)\Delta_{n}$ which can be obtained in the analogous manner. For $r_{j,n}^{\left(2\right)}$, we obtain
	\begin{align*}
	\left\|r_{j,n}^{\left(2\right)}\right\|_{p}
	&=\mathbf{E}_{\theta^{\star}}\left[\left|\frac{1}{p_{n}}\sum_{i=0}^{p_{n}-1}\left(i+1\right)\int_{j\Delta_{n}+ih_{n}}^{j\Delta_{n}+\left(i+1\right)h_{n}}\left(a\left(X_{s}\right)-a\left(X_{j\Delta_{n}+ih_{n}}\right)\right)\mathrm{d}w_{s}\right|^{p}\right]^{1/p}\\
	&\le C\left(p\right)\mathbf{E}_{\theta^{\star}}\left[\left|\frac{1}{p_{n}^{2}}\sum_{i=0}^{p_{n}-1}\left(i+1\right)^{2}\int_{j\Delta_{n}+ih_{n}}^{j\Delta_{n}+\left(i+1\right)h_{n}}\left\|a\left(X_{s}\right)-a\left(X_{j\Delta_{n}+ih_{n}}\right)\right\|^{2}\mathrm{d}s\right|^{p/2}\right]^{1/p}\\
	&\le C\left(p\right)\left(\sum_{i=0}^{p_{n}-1}\mathbf{E}_{\theta^{\star}}\left[\left|\int_{j\Delta_{n}+ih_{n}}^{j\Delta_{n}+\left(i+1\right)h_{n}}\left\|a\left(X_{s}\right)-a\left(X_{j\Delta_{n}+ih_{n}}\right)\right\|^{2}\mathrm{d}s\right|^{p/2}\right]^{2/p}\right)^{1/2}\\
	&\le C\left(p\right)\left(\sum_{i=0}^{p_{n}-1}h_{n}^{1-2/p}\mathbf{E}_{\theta^{\star}}\left[\int_{j\Delta_{n}+ih_{n}}^{j\Delta_{n}+\left(i+1\right)h_{n}}\left\|a\left(X_{s}\right)-a\left(X_{j\Delta_{n}+ih_{n}}\right)\right\|^{p}\mathrm{d}s\right]^{2/p}\right)^{1/2}\\
	&\le C\left(p\right)\left(\sum_{i=0}^{p_{n}-1}h_{n}^{1-2/p}\left(\int_{j\Delta_{n}+ih_{n}}^{j\Delta_{n}+\left(i+1\right)h_{n}}\mathbf{E}_{\theta^{\star}}\left[\left\|a\left(X_{s}\right)-a\left(X_{j\Delta_{n}+ih_{n}}\right)\right\|^{p}\right]\mathrm{d}s\right)^{2/p}\right)^{1/2}\\
	&\le C\left(p\right)\left(\sum_{i=0}^{p_{n}-1}h_{n}\left(\sup_{s\in\left[j\Delta_{n}+ih_{n},j\Delta_{n}+\left(i+1\right)h_{n}\right]}\mathbf{E}_{\theta^{\star}}\left[\left\|a\left(X_{s}\right)-a\left(X_{j\Delta_{n}+ih_{n}}\right)\right\|^{p}\right]\mathrm{d}s\right)^{2/p}\right)^{1/2}\\
	&\le C\left(p\right)\left(p_{n}h_{n}^2\right)^{1/2}\\
	&\le C\left(p\right)\Delta_{n}^{3/2}
	\end{align*}
	because of BDG inequality, H\"{o}lder's inequality, Fubini's theorem and the fact that $h_{n}=\Delta_{n}/p_{n}\le \Delta_{n}^2$, and the same evaluation can be proved for $s_{j,n}^{\left(2\right)}$. It also holds
	\begin{align*}
	\left\|r_{j,n}^{\left(3\right)}\right\|_{p}
	&=\frac{1}{p_{n}}\sum_{k=0}^{p_{n}-1}\left(k+1\right)\mathbf{E}_{\theta^{\star}}\left[\left|\int_{j\Delta_{n}+kh_{n}}^{j\Delta_{n}+\left(k+1\right)h_{n}}\left(b\left(X_{s}\right)-b\left(X_{j\Delta_{n}}\right)\right)\mathrm{d}s\right|^{p}\right]^{1/p}\\
	&\le \frac{C\left(p\right)}{p_{n}}\sum_{k=0}^{p_{n}-1}\left(k+1\right)\mathbf{E}_{\theta^{\star}}\left[\left(\int_{j\Delta_{n}+kh_{n}}^{j\Delta_{n}+\left(k+1\right)h_{n}}\left|b\left(X_{s}\right)-b\left(X_{j\Delta_{n}}\right)\right|\mathrm{d}s\right)^{p}\right]^{1/p}\\
	&\le \frac{C\left(p\right)}{p_{n}}\sum_{k=0}^{p_{n}-1}\left(k+1\right)h_{n}^{1-1/p}\mathbf{E}_{\theta^{\star}}\left[\int_{j\Delta_{n}+kh_{n}}^{j\Delta_{n}+\left(k+1\right)h_{n}}\left|b\left(X_{s}\right)-b\left(X_{j\Delta_{n}}\right)\right|^{p}\mathrm{d}s\right]^{1/p}\\
	&\le \frac{C\left(p\right)}{p_{n}}\sum_{k=0}^{p_{n}-1}\left(k+1\right)h_{n}^{1-1/p}\left(\int_{j\Delta_{n}+kh_{n}}^{j\Delta_{n}+\left(k+1\right)h_{n}}\mathbf{E}_{\theta^{\star}}\left[\left|b\left(X_{s}\right)-b\left(X_{j\Delta_{n}}\right)\right|^{p}\right]\mathrm{d}s\right)^{1/p}\\
	&\le \frac{C\left(p\right)}{p_{n}}\sum_{k=0}^{p_{n}-1}\left(k+1\right)h_{n}\left(\sup_{s\in\left[j\Delta_{n},\left(j+1\right)\Delta_{n}\right]}\mathbf{E}_{\theta^{\star}}\left[\left|b\left(X_{s}\right)-b\left(X_{j\Delta_{n}}\right)\right|^{p}\right]\right)^{1/p}\\
	&\le \frac{C\left(p\right)\Delta_{n}^{1/2}h_{n}}{p_{n}}\sum_{k=0}^{p_{n}-1}\left(k+1\right)\\
	&\le C\left(p\right)\Delta_{n}^{3/2}
	\end{align*}
	by H\"{o}lder's inequality and Fubini's theorem, and same evaluation holds for $s_{j,n}^{\left(3\right)}$: $\left\|s_{j,n}^{\left(3\right)}\right\|_{p}\le C\left(p\right)\Delta_{n}^{3/2}$. Hence we obtain the evaluation for $\left\|e_{j,n}\right\|_{p}$. 
	
	\noindent In the next place, we show (ii) holds. Note that it is sufficient to see only the moments for $r_{j,n}^{\left(1\right)}\zeta_{j+1,n}^{T}$ and $s_{j,n}^{\left(1\right)}\left(\zeta_{j+2,n}'\right)^{T}$ because H\"{o}lder's inequality and orthogonality are applicable for the others. We have the following expression for $r_{j,n}^{\left(1\right)}$ and $s_{j,n}^{\left(1\right)}$:
	\begin{align*}
	r_{j,n}^{\left(1\right)}&=\frac{1}{p_{n}}\sum_{i=0}^{p_{n}-1}\left(i+1\right)\left(a\left(X_{j\Delta_{n}+ih_{n}}\right)-a\left(X_{j\Delta_{n}}\right)\right)\int_{j\Delta_{n}+ih_{n}}^{j\Delta_{n}+\left(i+1\right)h_{n}}\mathrm{d}w_{s},\\
	s_{j,n}^{\left(1\right)}&=\frac{1}{p_{n}}\sum_{i=0}^{p_{n}-1}\left(p_{n}-i-1\right)\left(a\left(X_{\left(j+1\right)\Delta_{n}+ih_{n}}\right)-a\left(X_{j\Delta_{n}}\right)\right)\int_{\left(j+1\right)\Delta_{n}+ih_{n}}^{\left(j+1\right)\Delta_{n}+\left(i+1\right)h_{n}}\mathrm{d}w_{s}.
	\end{align*}
	Let us define for all $\ell=p_{n},\ldots,\left(k_{n}-2\right)p_{n}+p_{n}-1$, $\ell_{1}\left(\ell\right)=\left\lfloor\ell/p_{n}\right\rfloor$, and $\ell_{2}\left(\ell\right)=\ell-\ell_{1}\left(\ell\right)$,
	\begin{align*}
	\mathbb{D}_{\ell}^{n}
	&=\sum_{j=1}^{\ell_{1}\left(\ell\right)}\sum_{i=0}^{\ell_{2}\left(\ell\right)}\frac{i+1}{p_{n}}\mathbb{B}_{j}^{n}
	\left(a\left(X_{j\Delta_{n}+ih_{n}}\right)-a\left(X_{j\Delta_{n}}\right)\right)
	\left[\left(\int_{j\Delta_{n}+ih_{n}}^{j\Delta_{n}+\left(i+1\right)h_{n}}\mathrm{d}w_{s}\right)^{\otimes 2}\right],\\
	\mathbf{D}_{\ell}^{n}&=\sum_{j=1}^{\ell_{1}\left(\ell\right)}\sum_{i=0}^{\ell_{2}\left(\ell\right)}\frac{i+1}{p_{n}}\mathbb{B}_{j}^{n}
	\left(a\left(X_{j\Delta_{n}+ih_{n}}\right)-a\left(X_{j\Delta_{n}}\right)\right)
	\left[h_{n}I_{r}\right],
	\end{align*}
	and then we have 
	$\sum_{j=1}^{k_{n}-2}\mathbb{B}_{j}^{n}\left[r_{j,n}^{\left(1\right)}\left(\zeta_{j+1,n}\right)^{T}\right]=\mathbb{D}_{\left(k_{n}-2\right)p_{n}+p_{n}-1}^{n}.
	$
	We can easily observe that $\mathbb{D}_{\ell}^{n}-\mathbf{D}_{\ell}^{n}$ is a martingale with respect to $\left\{\mathcal{H}_{\ell_{1}\left(\ell\right),\ell_{2}\left(\ell\right)}^{n}\right\}$. Then Burkholder's inequality is applicable and it follows that
	\begin{align*}
	&\mathbf{E}_{\theta^{\star}}\left[\left|\mathbb{D}_{\left(k_{n}-2\right)p_{n}+p_{n}-1}^{n}-\mathbf{D}_{\left(k_{n}-2\right)p_{n}+p_{n}-1}^{n}\right|^{p}\right]\\
	&\le C\left(p\right)\mathbf{E}_{\theta^{\star}}\left[\left|\sum_{j=1}^{k_{n}-2}\sum_{i=0}^{p_{n}-1}\left(\frac{i+1}{p_{n}}\right)^{2}\left\|\mathbb{B}_{j}^{n}\right\|^{2}\left\|a\left(X_{j\Delta_{n}+ih_{n}}\right)-a\left(X_{j\Delta_{n}}\right)\right\|^{2}\right.\right.\\
	&\hspace{5cm}\left.\left.\times\left(\left|\int_{j\Delta_{n}+ih_{n}}^{j\Delta_{n}+\left(i+1\right)h_{n}}\mathrm{d}w_{s}\right|^{4}+r^{2}h_{n}^{2}\right)\right|^{p/2}\right]\\
	&\le C\left(p\right)\frac{n^{p/2}}{n}\sum_{j=1}^{k_{n}-2}\sum_{i=0}^{p_{n}-1}\mathbf{E}_{\theta^{\star}}\left[\left|\left\|\mathbb{B}_{j}^{n}\right\|^{2}\left\|a\left(X_{j\Delta_{n}+ih_{n}}\right)-a\left(X_{j\Delta_{n}}\right)\right\|^{2}\right.\right.\\
	&\hspace{5cm}\left.\left.\times\left(\left|\int_{j\Delta_{n}+ih_{n}}^{j\Delta_{n}+\left(i+1\right)h_{n}}\mathrm{d}w_{s}\right|^{4}+r^{2}h_{n}^{2}\right)\right|^{p/2}\right]\\
	&\le C\left(p\right)\frac{n^{p/2}}{n}\sum_{j=1}^{k_{n}-2}\sum_{i=0}^{p_{n}-1}\mathbf{E}_{\theta^{\star}}\left[\left\|a\left(X_{j\Delta_{n}+ih_{n}}\right)-a\left(X_{j\Delta_{n}}\right)\right\|^{2p}\left(\left|\int_{j\Delta_{n}+ih_{n}}^{j\Delta_{n}+\left(i+1\right)h_{n}}\mathrm{d}w_{s}\right|^{4}-r^{2}h_{n}^{2}\right)^{p}\right]^{1/2}\\
	&\le C\left(p\right)\frac{n^{p/2}}{n}\sum_{j=1}^{k_{n}-2}\sum_{i=0}^{p_{n}-1}\mathbf{E}_{\theta^{\star}}\left[\left\|a\left(X_{j\Delta_{n}+ih_{n}}\right)-a\left(X_{j\Delta_{n}}\right)\right\|^{4p}\right]^{1/4}\\
	&\hspace{3cm}\times \mathbf{E}_{\theta^{\star}}\left[\left(\left|\int_{j\Delta_{n}+ih_{n}}^{j\Delta_{n}+\left(i+1\right)h_{n}}\mathrm{d}w_{s}\right|^{4}+r^{2}h_{n}^{2}\right)^{2p}\right]^{1/4}\\
	&\le C\left(p\right)n^{p/2}\Delta_{n}^{p/2}h_{n}^{p}\\
	&= C\left(p\right)k_{n}^{p/2}\Delta_{n}^{p}h_{n}^{p/2}\\
	&\le C\left(p\right)k_{n}^{p/2}\Delta_{n}^{3p}.
	\end{align*}
	Hence we have $\left\|\mathbb{D}_{\left(k_{n}-2\right)p_{n}+p_{n}-1}^{n}-\mathbf{D}_{\left(k_{n}-2\right)p_{n}+p_{n}-1}^{n}\right\|_{p}\le C\left(p\right)k_{n}^{1/2}\Delta_{n}$. Furthermore, let us define
	\begin{align*}
	\mathbf{D}_{\ell}^{1,n}&=\sum_{j=1}^{\ell_{1}\left(\ell\right)}\sum_{i=0}^{\ell_{2}\left(\ell\right)}\frac{i+1}{p_{n}}\mathbb{B}_{j}^{n}
	\left(a\left(X_{j\Delta_{n}+ih_{n}}\right)-\mathbf{E}_{\theta^{\star}}\left[a\left(X_{j\Delta_{n}+ih_{n}}\right)|\mathcal{H}_{j}^{n}\right]\right)
	\left[h_{n}I_{r}\right],\\
	\mathbf{D}_{\ell}^{2,n}&=\sum_{j=1}^{\ell_{1}\left(\ell\right)}\sum_{i=0}^{\ell_{2}\left(\ell\right)}\frac{i+1}{p_{n}}\mathbb{B}_{j}^{n}
	\left(\mathbf{E}_{\theta^{\star}}\left[a\left(X_{j\Delta_{n}+ih_{n}}\right)|\mathcal{H}_{j}^{n}\right]-a\left(X_{j\Delta_{n}}\right)\right)
	\left[h_{n}I_{r}\right],
	\end{align*}
	and clearly we have $\mathbf{D}_{\ell}^{n}=\mathbf{D}_{\ell}^{1,n}+\mathbf{D}_{\ell}^{2,n}$. In addition, we see $\left\{\mathbf{D}_{jp_{n}+p_{n}-1}^{1,n}\right\}_{j=1,\ldots,k_{n}-2}$ is a martingale with respect to $\left\{\mathcal{H}_{j}^{n}\right\}_{j=1,\ldots,k_{n}-2}$, and then Burkholder's inequality leads to
	\begin{align*}
	&\mathbf{E}_{\theta^{\star}}\left[\left|\mathbf{D}_{\left(k_{n}-2\right)p_{n}+p_{n}-1}^{1,n}\right|^{p}\right]\\
	&\le \mathcal{C}\left(p\right)\mathbf{E}_{\theta^{\star}}\left[\left|\sum_{j=1}^{k_{n}-2}p_{n}\sum_{i=0}^{p_{n}-1}\left(\frac{i+1}{p_{n}}\right)^{2}\left\|\mathbb{B}_{j}^{n}\right\|^{2}
	\left\|a\left(X_{j\Delta_{n}+ih_{n}}\right)-\mathbf{E}_{\theta^{\star}}\left[a\left(X_{j\Delta_{n}+ih_{n}}\right)|\mathcal{H}_{j}^{n}\right]\right\|^{2}
	h_{n}^2\right|^{p/2}\right]\\
	&\le \mathcal{C}\left(p\right)n^{p/2}p_{n}^{p/2}\Delta_{n}^{p/2}h_{n}^{p}\\
	&\le C\left(p\right)k_{n}^{p/2}\Delta_{n}^{3p/2}\\
	&\le C\left(p\right)k_{n}^{p}\Delta_{n}^{2p}.
	\end{align*}
	Regarding $\mathbf{D}_{\ell}^{2,n}$, we have
	\begin{align*}
	&\mathbf{E}_{\theta^{\star}}\left[\left|\mathbf{D}_{\left(k_{n}-2\right)p_{n}+p_{n}-1}^{2,n}\right|^{p}\right]\\
	&=\mathbf{E}_{\theta^{\star}}\left[\left|\sum_{j=1}^{k_{n}-2}\sum_{i=0}^{p_{n}-1}\frac{i+1}{p_{n}}\mathbb{B}_{j}^{n}
	\left(\mathbf{E}_{\theta^{\star}}\left[a\left(X_{j\Delta_{n}+ih_{n}}\right)|\mathcal{H}_{j}^{n}\right]-a\left(X_{j\Delta_{n}}\right)\right)
	\left[h_{n}I_{r}\right]\right|^{p}\right]\\
	&\le \mathbf{E}_{\theta^{\star}}\left[\left|\sum_{j=1}^{\ell_{1}}\sum_{i=0}^{\ell_{2}}\left\|\mathbb{B}_{j}^{n}\right\|
	\left\|\mathbf{E}_{\theta^{\star}}\left[a\left(X_{j\Delta_{n}+ih_{n}}\right)|\mathcal{H}_{j}^{n}\right]-a\left(X_{j\Delta_{n}}\right)\right\|\right|^{p}\right]h_{n}^{p}\\
	&\le C\left(p\right)n^{p}h_{n}^{p}\Delta_{n}^p\\
	&= C\left(p\right)k_{n}^{p}\Delta_{n}^{2p},
	\end{align*}
	since $\left\|\mathbf{E}_{\theta^{\star}}\left[a\left(X_{j\Delta_{n}+ih_{n}}\right)|\mathcal{H}_{j}^{n}\right]-a\left(X_{j\Delta_{n}}\right)\right\|\le C\Delta_{n}\left(1+\left|X_{j\Delta_{n}}\right|\right)^{C}$.
	The same evaluation holds for $s_{j,n}^{\left(1\right)}$, and hence we obtain the result.
	
	Finally we check that (iii) holds. It is only necessary to verify it for $r_{j,n}^{\left(1\right)}$ and $s_{j,n}^{\left(1\right)}$, and we show with respect to $r_{j,n}^{\left(1\right)}$. Since $\left\{\sum_{k=1}^{j}\mathbb{C}_{k,n}\left[r_{k,n}^{\left(1\right)}\right]\right\}$ for $\ell\le k_{n}-2$ is a martingale with respect to $\left\{\mathcal{H}_{j}^{n}\right\}$, we can utilise Burkholder's inequality and then
	\begin{align*}
	\mathbf{E}_{\theta^{\star}}\left[\left|\sum_{j=1}^{k_{n}-2}\mathbb{C}_{j,n}\left[r_{j,n}^{\left(1\right)}\right]\right|^{p}\right]&\le C\left(p\right)k_{n}^{p/2-1}\sum_{j=1}^{k_{n}-2}\mathbf{E}_{\theta^{\star}}\left[\left|r_{j,n}^{\left(1\right)}\right|^{2p}\right]^{1/2}\\
	&\le C\left(p\right)k_{n}^{p/2}\Delta_{n}^{p}
	\end{align*}
	and we can have the same evaluation for $s_{j,n}^{\left(1\right)}$.
\end{proof}

\begin{remark}
	When the evaluation $\left\|e_{j,n}\right\|_{p}\le C\left(p\right)\Delta_{n}$ is sufficient, then we can abbreviate $\Delta_{n}b\left(X_{j\Delta_{n}}\right)$ in the right hand side.
\end{remark}

\begin{lemma}\label{MomentLambda}
	\begin{itemize}
		\item[(a)]For all $p\ge 1$, there exists $C\left(p\right)>0$ such that for all $j=0,\ldots,k_{n}-1$ and $n\in\mathbf{N}$,
		\begin{align*}
			\left\|\lm{\varepsilon}{j}\right\|_{p}&\le C\left(p\right)p_{n}^{-1/2}.
		\end{align*}
		\item[(b)]For all $p\ge 1$, there exists $C\left(p\right)>0$ such that for all $n\in\mathbf{N}$
		\begin{align*}
		\left\|\hat{\Lambda}_{n}-\Lambda_{\star}\right\|_{p}\le C\left(p\right)\left(h_{n}+\frac{1}{\sqrt{n}}\right).
		\end{align*}
	\end{itemize}
\end{lemma}

\begin{proof}
	(a) Because of H\"{o}lder's inequality, it is enough to evaluate it in the case where $p$ is an even integer. We easily obtain
	\begin{align*}
		\mathbf{E}_{\theta^{\star}}\left[\left|\lm{\varepsilon}{j}\right|^{p}\right]&\le \sum_{\ell=1}^{d}\mathbf{E}_{\theta^{\star}}\left[\left|\lm{\varepsilon}{j}^{\left(\ell\right)}\right|^{p}\right]\\
		&=\frac{1}{p_{n}^{p}}\sum_{\ell=1}^{d}\sum_{i_{1}=0}^{p_{n}-1}\cdots\sum_{i_{p/2}=0}^{p_{n}-1}
		\mathbf{E}_{\theta^{\star}}\left[\left|\varepsilon_{j\Delta_{n}+i_{1}h_{n}}^{\left(\ell\right)}\right|^{2}\cdots\left|\varepsilon_{j\Delta_{n}+i_{p/2}h_{n}}^{\left(\ell\right)}\right|^{2}
		\right]\\
		&\le C\left(p\right)p_{n}^{-p/2}
	\end{align*}
	for [A5].
	
	\noindent (b) As (a), it is enough to evaluate in the case where $p$ is an even integer. Then we have
	\begin{align*}
	\left\|\hat{\Lambda}_{n}-\Lambda_{\star}\right\|_{p}
	&=\mathbf{E}_{\theta^{\star}}\left[\left\|\frac{1}{2n}\sum_{i=1}^{n}\left(Y_{ih_{n}}-Y_{\left(i-1\right)h_{n}}\right)^{\otimes 2}
	-\Lambda_{\star}\right\|^{p}\right]^{1/p}\\
	&\le \mathbf{E}_{\theta^{\star}}\left[\left\|\frac{1}{2n}\sum_{i=1}^{n}\left(X_{ih_{n}}-X_{\left(i-1\right)h_{n}}\right)^{\otimes 2}
	\right\|^{p}\right]^{1/p}\\
	&\qquad+\mathbf{E}_{\theta^{\star}}\left[\left\|\frac{1}{2n}\sum_{i=1}^{n}\left(X_{ih_{n}}-X_{\left(i-1\right)h_{n}}\right)\left(\varepsilon_{ih_{n}}-\varepsilon_{\left(i-1\right)h_{n}}\right)^{T}\left(\Lambda_{\star}\right)^{1/2}\right\|^{p}\right]^{1/p}\\
	&\qquad+\mathbf{E}_{\theta^{\star}}\left[\left\|\frac{1}{2n}\sum_{i=1}^{n}\left[\left(\Lambda_{\star}\right)^{1/2}\left(\varepsilon_{ih_{n}}-\varepsilon_{\left(i-1\right)h_{n}}\right)\right]^{\otimes 2}
	-\Lambda_{\star}\right\|^{p}\right]^{1/p}\\
	&\le \frac{1}{2n}\sum_{i=1}^{n}\mathbf{E}_{\theta^{\star}}\left[\left|X_{ih_{n}}-X_{\left(i-1\right)h_{n}}
	\right|^{2p}\right]^{1/p}\\
	&\qquad+\frac{C\left(p\right)}{2n}\mathbf{E}_{\theta^{\star}}\left[\left\|\sum_{i=1}^{n}\left(X_{ih_{n}}-X_{\left(i-1\right)h_{n}}\right)\varepsilon_{ih_{n}}^{T}\right\|^{p}\right]^{1/p}\\
	&\qquad+\frac{C\left(p\right)}{2n}\mathbf{E}_{\theta^{\star}}\left[\left\|\sum_{i=1}^{n}\left(X_{ih_{n}}-X_{\left(i-1\right)h_{n}}\right)\varepsilon_{\left(i-1\right)h_{n}}^{T}\right\|^{p}\right]^{1/p}\\
	&\qquad+C\left(p\right)\mathbf{E}_{\theta^{\star}}\left[\left\|\frac{1}{2n}\sum_{i=1}^{n}\varepsilon_{ih_{n}}^{\otimes 2}
	-\frac{1}{2}I_{d}\right\|^{p}\right]^{1/p}\\
	&\qquad +C\left(p\right)\mathbf{E}_{\theta^{\star}}\left[\left\|\frac{1}{2n}\sum_{i=1}^{n}\varepsilon_{\left(i-1\right)h_{n}}^{\otimes 2}
	-\frac{1}{2}I_{d}\right\|^{p}\right]^{1/p}\\
	&\qquad +\frac{C\left(p\right)}{n}\mathbf{E}_{\theta^{\star}}\left[\left\|\sum_{i=1}^{n}\varepsilon_{ih_{n}}\varepsilon_{\left(i-1\right)h_{n}}^{T}\right\|^{p}\right]^{1/p}.
	\end{align*}
	The first term of the right hand side has the evaluation
	\begin{align*}
	\frac{1}{2n}\sum_{i=1}^{n}\mathbf{E}_{\theta^{\star}}\left[\left|X_{ih_{n}}-X_{\left(i-1\right)h_{n}}
	\right|^{2p}\right]^{1/p}\le C\left(p\right)h_{n}.
	\end{align*}
	We can evaluate the second term of the right hand side
	\begin{align*}
	&\mathbf{E}_{\theta^{\star}}\left[\left\|\sum_{i=1}^{n}\left(X_{ih_{n}}-X_{\left(i-1\right)h_{n}}\right)\varepsilon_{ih_{n}}^{T}\right\|^{p}\right]\\
	&=\mathbf{E}_{\theta^{\star}}\left[\sum_{i_{1}}\cdots\sum_{i_{p/2}}\left\|\left(X_{i_{1}h_{n}}-X_{\left(i_{1}-1\right)h_{n}}\right)\varepsilon_{i_{1}h_{n}}^{T}\right\|^{2}\cdots\left\|\left(X_{i_{p/2}h_{n}}-X_{\left(i_{p/2}-1\right)h_{n}}\right)\varepsilon_{i_{p/2}h_{n}}^{T}\right\|^{2}\right]\\
	&\le \sum_{i_{1}}\cdots\sum_{i_{p/2}}\mathbf{E}_{\theta^{\star}}\left[\left\|\left(X_{i_{1}h_{n}}-X_{\left(i_{1}-1\right)h_{n}}\right)\varepsilon_{i_{1}h_{n}}^{T}\right\|^{2}\cdots\left\|\left(X_{i_{p/2}h_{n}}-X_{\left(i_{p/2}-1\right)h_{n}}\right)\varepsilon_{i_{p/2}h_{n}}^{T}\right\|^{2}\right]\\
	&\le \sum_{i_{1}}\cdots\sum_{i_{p/2}}C\left(p\right)h_{n}^{p/2}\\
	&\le C\left(p\right)\left(nh_{n}\right)^{p/2}
	\end{align*}
	and hence
	\begin{align*}
	\frac{C\left(p\right)}{2n}\mathbf{E}_{\theta^{\star}}\left[\left\|\sum_{i=1}^{n}\left(X_{ih_{n}}-X_{\left(i-1\right)h_{n}}\right)\varepsilon_{ih_{n}}^{T}\right\|^{p}\right]^{1/p}\le  C\left(p\right)\sqrt{\frac{h_{n}}{n}}.
	\end{align*}
	The evaluation for the third term can be obtained in the same manner. For the fourth term, we have
	\begin{align*}
	&C\left(p\right)\mathbf{E}_{\theta^{\star}}\left[\left\|\frac{1}{2n}\sum_{i=1}^{n}\varepsilon_{ih_{n}}^{\otimes 2}
	-\frac{1}{2}I_{d}\right\|^{p}\right]^{1/p}\\
	&=C\left(p\right)\mathbf{E}_{\theta^{\star}}\left[\left\|\frac{1}{2n}\sum_{i=1}^{n}\left(\varepsilon_{ih_{n}}^{\otimes 2}
	-I_{d}\right)\right\|^{p}\right]^{1/p}\\
	&=\frac{C\left(p\right)}{n}\mathbf{E}_{\theta^{\star}}\left[\sum_{i_{1}}\cdots\sum_{i_{p/2}}\left\|\varepsilon_{i_{1}h_{n}}^{\otimes 2}
	-I_{d}\right\|^{2}\cdots\left\|\varepsilon_{i_{p/2}h_{n}}^{\otimes 2}
	-I_{d}\right\|^{2}\right]^{1/p}\\
	&\le \frac{C\left(p\right)}{\sqrt{n}},
	\end{align*}
	and the same evaluation holds for the fifth term. Finally we obtain
	\begin{align*}
	&\frac{C\left(p\right)}{n}\mathbf{E}_{\theta^{\star}}\left[\left\|\sum_{i=1}^{n}\varepsilon_{ih_{n}}\varepsilon_{\left(i-1\right)h_{n}}^{T}\right\|^{p}\right]^{1/p}\\
	&=\frac{C\left(p\right)}{n}\mathbf{E}_{\theta^{\star}}\left[\sum_{i_{1}}\cdots\sum_{i_{p/2}}\left\|\varepsilon_{i_{1}h_{n}}\varepsilon_{\left(i_{1}-1\right)h_{n}}^{T}\right\|^{2}\cdots\left\|\varepsilon_{i_{p/2}h_{n}}\varepsilon_{\left(i_{p/2}-1\right)h_{n}}^{T}\right\|^{2}\right]^{1/p}\\
	&\le \frac{C\left(p\right)}{\sqrt{n}}.
	\end{align*}
	Hence the evaluation for $L^{p}$-norm stated above holds.
\end{proof}

\begin{lemma}\label{EvalFuncLambda}
	For every function $f$ such that $f\in C^{1}\left(\R^{d}\times \Xi;\ \R\right)$ and all the elements of $f$ and the derivatives are polynomial growth with respect to $x$ uniformly in $\vartheta$,
	\begin{align*}
	&\mathbf{E}_{\theta^{\star}}\left[k_{n}^{\epsilon_{1}}\sup_{\vartheta\in\Xi}\left|\frac{1}{k_{n}}\sum_{j=1}^{k_{n}-2}\left(f\left(\lm{Y}{j-1},\vartheta,\hat{\Lambda}_{n}\right)-f\left(\lm{Y}{j-1},\vartheta,\Lambda_{\star}\right)\right)\right|^{p}\right]^{1/p}\\
	&\le C\left(p\right)\left(n^{1/2}h_{n}^{3/2}+\frac{1}{\sqrt{p_{n}}}\right).
	\end{align*}
\end{lemma}

\begin{proof}
	We have
	\begin{align*}
	&\mathbf{E}_{\theta^{\star}}\left[k_{n}^{\epsilon_{1}}\sup_{\vartheta\in\Xi}\left|\frac{1}{k_{n}}\sum_{j=1}^{k_{n}-2}\left(f\left(\lm{Y}{j-1},\vartheta,\hat{\Lambda}_{n}\right)-f\left(\lm{Y}{j-1},\vartheta,\Lambda_{\star}\right)\right)\right|^{p}\right]^{1/p}\\
	&\le k_{n}^{\epsilon_{1}}\left( \frac{1}{k_{n}}\sum_{j=1}^{k_{n}-2}\mathbf{E}_{\theta^{\star}}\left[\sup_{\vartheta\in\Xi}\left|f\left(\lm{Y}{j-1},\vartheta,\hat{\Lambda}_{n}\right)-f\left(\lm{Y}{j-1},\vartheta,\Lambda_{\star}\right)\right|^{p}\right]\right)^{1/p}\\
	&\le k_{n}^{\epsilon_{1}}\left( \frac{1}{k_{n}}\sum_{j=1}^{k_{n}-2}\mathbf{E}_{\theta^{\star}}\left[C\left(1+\left|\lm{Y}{j-1}\right|\right)^C\left\|\hat{\Lambda}_{n}-\Lambda_{\star}\right\|^{p}\right]\right)^{1/p}\\
	&\le C\left(p\right)k_{n}^{\epsilon_{1}}\mathbf{E}_{\theta^{\star}}\left[\left\|\hat{\Lambda}_{n}-\Lambda_{\star}\right\|^{2p}\right]^{1/2p}\\
	&\le C\left(p\right)\left(k_{n}^{\epsilon_{1}}h_{n}+\frac{k_{n}^{\epsilon_{1}}}{\sqrt{n}}\right)\\
	&\le C\left(p\right)\left(n^{1/2}h_{n}^{3/2}+\frac{1}{\sqrt{p_{n}}}\right).
	\end{align*}
\end{proof}

\subsection{LAN for the quasi-likelihoods and proof for the main theorem}
To prove the main theorem, we set some additional preliminary lemmas. 
Before the discussion, let us define the statistical random fields:
\begin{align*}
\mathbb{Y}_{1,n}^{\tau}\left(\alpha;\vartheta^{\star}\right)&=\frac{1}{k_{n}}\left(\mathbb{H}_{1,n}^{\tau}\left(\alpha;\hat{\Lambda}_{n}\right)-\mathbb{H}_{1,n}^{\tau}\left(\alpha^{\star};\hat{\Lambda}_{n}\right)\right)\\
&=-\frac{1}{2k_{n}}\sum_{j=1}^{k_{n}-2}
\left(\left(A_{n}^{\tau}\left(\lm{Y}{j-1},\alpha,\hat{\Lambda}_{n}\right)^{-1}-A_{n}^{\tau}\left(\lm{Y}{j-1},\alpha^{\star},\hat{\Lambda}_{n}\right)^{-1}\right)\left[\left(\lm{Y}{j+1}-\lm{Y}{j}\right)^{\otimes 2}\right]\right.\\
&\hspace{3cm}\left.\times \left(\frac{2}{3}\Delta_{n}\right)^{-1}+\log\frac{\det A_{n}^{\tau}\left(\lm{Y}{j-1},\alpha,\hat{\Lambda}_{n}\right)}{\det A_{n}^{\tau}\left(\lm{Y}{j-1},\alpha^{\star},\hat{\Lambda}_{n}\right)}\right),\\
\mathbb{Y}_{2,n}^{\mathrm{ML}}\left(\beta;\vartheta^{\star}\right)
&=\frac{1}{k_{n}\Delta_{n}}\left(\mathbb{H}_{2,n}\left(\beta;\hat{\alpha}_{n}\right)-\mathbb{H}_{2,n}\left(\beta^{\star};\hat{\alpha}_{n}\right)\right)\\
&=\frac{1}{k_{n}\Delta_{n}}\left(\sum_{j=1}^{k_{n}-2}A\left(\lm{Y}{j-1},\hat{\alpha}_{n}\right)^{-1}\left[b\left(\lm{Y}{j-1},\beta\right)-b\left(\lm{Y}{j-1},\beta^{\star}\right),\lm{Y}{j+1}-\lm{Y}{j}\right]\right.\\
&\hspace{2cm}\left.-\frac{\Delta_{n}}{2}\sum_{j=1}^{k_{n}-2}A\left(\lm{Y}{j-1},\hat{\alpha}_{n}\right)^{-1}
\left[b\left(\lm{Y}{j-1},\beta\right)^{\otimes 2}-b\left(\lm{Y}{j-1},\beta^{\star}\right)^{\otimes 2}\right]\right),\\
\mathbb{Y}_{2,n}^{\mathrm{Bayes}}\left(\beta;\vartheta^{\star}\right)
&=\frac{1}{k_{n}\Delta_{n}}\left(\mathbb{H}_{2,n}\left(\beta;\tilde{\alpha}_{n}\right)-\mathbb{H}_{2,n}\left(\beta^{\star};\tilde{\alpha}_{n}\right)\right)\\
&=\frac{1}{k_{n}\Delta_{n}}\left(\sum_{j=1}^{k_{n}-2}A\left(\lm{Y}{j-1},\tilde{\alpha}_{n}\right)^{-1}\left[b\left(\lm{Y}{j-1},\beta\right)-b\left(\lm{Y}{j-1},\beta^{\star}\right),\lm{Y}{j+1}-\lm{Y}{j}\right]\right.\\
&\hspace{2cm}\left.-\frac{\Delta_{n}}{2}\sum_{j=1}^{k_{n}-2}A\left(\lm{Y}{j-1},\tilde{\alpha}_{n}\right)^{-1}
\left[b\left(\lm{Y}{j-1},\beta\right)^{\otimes 2}-b\left(\lm{Y}{j-1},\beta^{\star}\right)^{\otimes 2}\right]\right).
\end{align*}
We give the locally asymptotic quadratic at $\vartheta^{\star}\in\Xi$ for $u_{1}\in\R^{m_{1}}$ and $u_{2}\in\R^{m_{2}}$,
\begin{align*}
	\mathbb{Z}_{1,n}^{\tau}\left(u_{1};\hat{\Lambda}_{n},\alpha^{\star}\right)&:=\exp\left(\Delta_{1,n}^{\tau}\left(\vartheta^{\star}\right)\left[u_{1}\right]-\frac{1}{2}\Gamma_{1}^{\tau}\left(\vartheta^{\star}\right)\left[u_{1}^{\otimes2}\right]+r_{1,n}^{\tau}\left(u;\vartheta^{\star}\right)\right),\\
	\mathbb{Z}_{2,n}^{\mathrm{ML}}\left(u_{2};\hat{\alpha}_{n},\beta^{\star}\right)&:=\exp\left(
	\Delta_{2,n}^{\mathrm{ML}}\left(\vartheta^{\star}\right)\left[u_{2}\right]
	-\frac{1}{2}\Gamma_{2}^{\mathrm{ML}}\left(\vartheta^{\star}\right)\left[u_{2}^{\otimes2}\right]
	+r_{2,n}^{\mathrm{ML}}\left(u;\vartheta^{\star}\right)\right),\\
	\mathbb{Z}_{2,n}^{\mathrm{Bayes}}\left(u_{2};\tilde{\alpha}_{n},\beta^{\star}\right)&:=\exp\left(\Delta_{2,n}^{\mathrm{Bayes}}\left(\vartheta^{\star}\right)\left[u_{2}\right]
	-\frac{1}{2}\Gamma_{2}^{\mathrm{Bayes}}\left(\vartheta^{\star}\right)\left[u_{2}^{\otimes2}\right]
	+r_{2,n}^{\mathrm{Bayes}}\left(u;\vartheta^{\star}\right)\right),
\end{align*}
where
\begin{align*}
	\Delta_{1,n}^{\tau}\left(\vartheta^{\star}\right)\left[u_{1}\right]&:=-\frac{1}{2k_{n}^{1/2}}\sum_{j=1}^{k_{n}-2}\left(\partial_{\alpha}A_{n}^{\tau}\left(\lm{Y}{j-1},\alpha^{\star},\hat{\Lambda}_{n}\right)^{-1}\left[u_{1},\left(\lm{Y}{j+1}-\lm{Y}{j}\right)^{\otimes 2}\right]\left(\frac{2\Delta_{n}}{3}\right)^{-1}\right.\\
	&\hspace{3cm}\left.+\partial_{\alpha}\log\frac{\det A_{n}^{\tau}\left(\lm{Y}{j-1},\alpha^{\star},\hat{\Lambda}_{n}\right)}{\det A_{n}^{\tau}\left(\lm{Y}{j-1},\alpha^{\star},\hat{\Lambda}_{n}\right)}\left[u_{1}\right]\right),\\
	\Delta_{2,n}^{\mathrm{ML}}\left(\vartheta^{\star}\right)\left[u_{2}\right]
	&:=\frac{1}{\left(k_{n}\Delta_{n}\right)^{1/2}}\sum_{j=1}^{k_{n}-2}A\left(\lm{Y}{j-1},\hat{\alpha}_{n}\right)^{-1}\\
	&\hspace{3cm}\left[\partial_{\beta}b\left(\lm{Y}{j-1},\beta^{\star}\right)u_{2},\lm{Y}{j+1}-\lm{Y}{j}-\Delta_{n}b\left(\lm{Y}{j-1},\beta^{\star}\right)\right],\\
	\Delta_{2,n}^{\mathrm{Bayes}}\left(\vartheta^{\star}\right)\left[u_{2}\right]
	&:=\frac{1}{\left(k_{n}\Delta_{n}\right)^{1/2}}\sum_{j=1}^{k_{n}-2}A\left(\lm{Y}{j-1},\tilde{\alpha}_{n}\right)^{-1}\\
	&\hspace{3cm}\left[\partial_{\beta}b\left(\lm{Y}{j-1},\beta^{\star}\right)u_{2},\lm{Y}{j+1}-\lm{Y}{j}-\Delta_{n}b\left(\lm{Y}{j-1},\beta^{\star}\right)\right]
\end{align*}
and
\begin{align*}
	&\Gamma_{1,n}^{\tau}\left(\alpha;\vartheta^{\star}\right)\left[u_{1}^{\otimes 2}\right]\\
	&:=\frac{1}{2k_{n}}\sum_{j=1}^{k_{n}-2}\left(\partial_{\alpha}^{2}A_{n}^{\tau}\left(\lm{Y}{j-1},\alpha,\hat{\Lambda}_{n}\right)^{-1}\left[u_{1}^{\otimes 2},\left(\lm{Y}{j+1}-\lm{Y}{j}\right)^{\otimes 2}\right]\left(\frac{2\Delta_{n}}{3}\right)^{-1}\right.\\
	&\hspace{3cm}\left.+\partial_{\alpha}^{2}\log\frac{\det A_{n}^{\tau}\left(\lm{Y}{j-1},\alpha,\hat{\Lambda}_{n}\right)}{\det A_{n}^{\tau}\left(\lm{Y}{j-1},\alpha^{\star},\hat{\Lambda}_{n}\right)}\left[u_{1}^{\otimes 2}\right]\right),\\
	&\Gamma_{2,n}^{\mathrm{ML}}\left(\beta;\vartheta^{\star}\right)\left[u_{2}^{\otimes 2}\right]\\
	&:=\frac{1}{k_{n}\Delta_{n}}\sum_{j=1}^{k_{n}-2}
	A\left(\lm{Y}{j-1},\hat{\alpha}_{n}\right)^{-1}\left[\partial_{\beta}b\left(\lm{Y}{j-1},\beta\right)\left[u_{2}\right],\Delta_{n}b\left(\lm{Y}{j-1},\beta\right)\left[u_{2}\right]\right]\\
	&\qquad-\frac{1}{k_{n}\Delta_{n}}\sum_{j=1}^{k_{n}-2}
	A\left(\lm{Y}{j-1},\hat{\alpha}_{n}\right)^{-1}\left[\partial_{\beta}^{2}b\left(\lm{Y}{j-1},\beta\right)\left[u_{2}^{\otimes2}\right],\lm{Y}{j+1}-\lm{Y}{j}-\Delta_{n}b\left(\lm{Y}{j-1},\beta\right)\right],\\
	&\Gamma_{2,n}^{\mathrm{Bayes}}\left(\beta;\vartheta^{\star}\right)\left[u_{2}^{\otimes 2}\right]\\
	&:=\frac{1}{k_{n}\Delta_{n}}\sum_{j=1}^{k_{n}-2}
	A\left(\lm{Y}{j-1},\tilde{\alpha}_{n}\right)^{-1}\left[\partial_{\beta}b\left(\lm{Y}{j-1},\beta\right)\left[u_{2}\right],\Delta_{n}b\left(\lm{Y}{j-1},\beta\right)\left[u_{2}\right]\right]\\
	&\qquad-\frac{1}{k_{n}\Delta_{n}}\sum_{j=1}^{k_{n}-2}
	A\left(\lm{Y}{j-1},\tilde{\alpha}_{n}\right)^{-1}\left[\partial_{\beta}^{2}b\left(\lm{Y}{j-1},\beta\right)\left[u_{2}^{\otimes2}\right],\lm{Y}{j+1}-\lm{Y}{j}-\Delta_{n}b\left(\lm{Y}{j-1},\beta\right)\right],
\end{align*}
and
\begin{align*}
	&\Gamma_{1}^{\tau}\left(\vartheta^{\star}\right)\left[u_{1}^{\otimes 2}\right]\\
	&:=\frac{1}{2}\left.\int_{\R^{d}}\left(\partial_{\alpha}^{2}A^{\tau}\left(x,\alpha,\Lambda_{\star}\right)^{-1}\left[u_{1}^{\otimes 2},A^{\tau}\left(x,\alpha^{\star},\Lambda_{\star}\right)\right]+\partial_{\alpha}^{2}\log\frac{\det A^{\tau}\left(x,\alpha,\Lambda_{\star}\right)}{\det A^{\tau}\left(x,\alpha^{\star},\Lambda_{\star}\right)}\left[u_{1}^{\otimes 2}\right]\right)
	\right|_{\alpha=\alpha^{\star}}\\
	&\hspace{2cm}\times\nu\left(\mathrm{d}x\right),\\
	&\Gamma_{2}\left(\vartheta^{\star}\right)\left[u_{2}^{\otimes 2}\right]\\
	&:=\frac{1}{2}\int_{\R^{d}}\left(A\left(x,\alpha\right)^{-1}\left[\partial_{\beta}b\left(x,\beta^{\star}\right)\left[u_{2}\right],\partial_{\beta}b\left(x,\beta^{\star}\right)\left[u_{2}\right]\right]\right)
	\nu\left(\mathrm{d}x\right),
\end{align*}
and
\begin{align*}
	r_{1,n}^{\tau}\left(u;\vartheta^{\star}\right)&:=\int_{0}^{1}\left(1-s\right)\left\{\Gamma_{1}^{\tau}\left(\vartheta^{\star}\right)\left[u_{1}^{\otimes2}\right]-\Gamma_{1,n}^{\tau}\left(\alpha^{\star}+sk_{n}^{-1/2}u_{1};\vartheta^{\star}\right)\left[u_{1}^{\otimes2}\right]\right\}\mathrm{d}s,\\
	r_{2,n}^{\mathrm{ML}}\left(u;\vartheta^{\star}\right)&:=\int_{0}^{1}\left(1-s\right)\left\{\Gamma_{2}^{\mathrm{ML}}\left(\vartheta^{\star}\right)\left[u_{2}^{\otimes2}\right]-\Gamma_{2,n}^{\mathrm{ML}}\left(\beta^{\star}+sT_{n}^{-1/2}u_{2};\vartheta^{\star}\right)\left[u_{2}^{\otimes2}\right]\right\}\mathrm{d}s,\\
	r_{2,n}^{\mathrm{Bayes}}\left(u;\vartheta^{\star}\right)&:=\int_{0}^{1}\left(1-s\right)\left\{\Gamma_{2}^{\mathrm{Bayes}}\left(\vartheta^{\star}\right)\left[u_{2}^{\otimes2}\right]-\Gamma_{2,n}^{\mathrm{Bayes}}\left(\beta^{\star}+sT_{n}^{-1/2}u_{2};\vartheta^{\star}\right)\left[u_{2}^{\otimes2}\right]\right\}\mathrm{d}s.
\end{align*}

We evaluate the moments of these random variables and fields in the following lemmas.

\begin{lemma}\label{alphaScore}
\begin{enumerate}
\item[(a)] For every $p>1$,
\begin{align*}
\sup_{n\in\mathbf{N}}\mathbf{E}_{\theta^{\star}}\left[\left|\Delta_{1,n}^{\tau}\left(\vartheta^{\star}\right)\right|^{p}\right]<\infty.
\end{align*}
\item[(b)] 
Let $\epsilon_{1}=\epsilon_{0}/2$. Then for every $p>0$,
\begin{align*}
\sup_{n\in\mathbf{N}}\mathbf{E}_{\theta^{\star}}\left[\left(\sup_{\alpha\in\Theta_{1}}k_{n}^{\epsilon_{1}}\left|\mathbb{Y}_{1,n}^{\tau}\left(\alpha;\vartheta^{\star}\right)-\mathbb{Y}_{1}^{\tau}\left(\alpha;\vartheta^{\star}\right)\right|\right)^{p}
\right]<\infty.
\end{align*}
\end{enumerate}
\end{lemma}

\begin{proof}
We start with the proof for (a).
By Lemma \ref{ExpansionLM}, we obtain a decomposition
\begin{align*}
	\Delta_{1,n}^{\tau}\left(\vartheta^{\star}\right)\left[u_{1}\right]
	=M_{1,n}^{\tau}+R_{1,n}^{\tau\left(1\right)}+R_{1,n}^{\tau\left(2\right)}+R_{1,n}^{\tau\left(3\right)}
\end{align*}
for
\begin{align*}
	M_{1,n}^{\tau}&:=-\frac{1}{2k_{n}^{1/2}}\sum_{j=1}^{k_{n}-2}\left(\partial_{\alpha}A_{n}^{\tau}\left(X_{j\Delta_{n}},\alpha^{\star},\Lambda_{\star}\right)^{-1}\left[u_{1},\widehat{A_{j,n}^{\tau}}\right]+\partial_{\alpha}\log\det A_{n}^{\tau}\left(X_{j\Delta_{n}},\alpha^{\star},\Lambda_{\star}\right)\left[u_{1}\right]\right),\\
	R_{1,n}^{\tau^{\left(1\right)}}&:=\frac{1}{2k_{n}^{1/2}}\sum_{j=1}^{k_{n}-2}\left(\partial_{\alpha}A_{n}^{\tau}\left(X_{j\Delta_{n}},\alpha^{\star},\Lambda_{\star}\right)^{-1}\left[u_{1},\widehat{A_{j,n}^{\tau}}\right]+\partial_{\alpha}\log\det A_{n}^{\tau}\left(X_{j\Delta_{n}},\alpha^{\star},\Lambda_{\star}\right)\left[u_{1}\right]\right)\\
	&\qquad-\frac{1}{2k_{n}^{1/2}}\sum_{j=1}^{k_{n}-2}\left(\partial_{\alpha}A_{n}^{\tau}\left(\lm{Y}{j-1},\alpha^{\star},\Lambda_{\star}\right)^{-1}\left[u_{1},\widehat{A_{j,n}^{\tau}}\right]+\partial_{\alpha}\log\det A_{n}^{\tau}\left(\lm{Y}{j-1},\alpha^{\star},\Lambda_{\star}\right)\left[u_{1}\right]\right),\\
	R_{1,n}^{\tau\left(2\right)}
	&:=\frac{1}{2k_{n}^{1/2}}\sum_{j=1}^{k_{n}-2}\left(\partial_{\alpha}A_{n}^{\tau}\left(\lm{Y}{j-1},\alpha^{\star},\Lambda_{\star}\right)^{-1}\left[u_{1},\widehat{A_{j,n}^{\tau}}\right]+\partial_{\alpha}\log\det A_{n}^{\tau}\left(\lm{Y}{j-1},\alpha^{\star},\Lambda_{\star}\right)\left[u_{1}\right]\right)\\
	&\qquad-\frac{1}{2k_{n}^{1/2}}\sum_{j=1}^{k_{n}-2}\left(\partial_{\alpha}A_{n}^{\tau}\left(\lm{Y}{j-1},\alpha^{\star},\hat{\Lambda}_{n}\right)^{-1}\left[u_{1},\widehat{A_{j,n}^{\tau}}\right]+\partial_{\alpha}\log\det A_{n}^{\tau}\left(\lm{Y}{j-1},\alpha^{\star},\hat{\Lambda}_{n}\right)\left[u_{1}\right]\right),\\
	R_{1,n}^{\tau\left(3\right)}
	&:= -\frac{1}{2k_{n}^{1/2}}\sum_{j=1}^{k_{n}-2}\partial_{\alpha}A_{n}^{\tau}\left(\lm{Y}{3j+i-1},\alpha^{\star},\hat{\Lambda}_{n}\right)^{-1}\left[u_{1},\frac{3}{2\Delta_{n}}\left(\lm{Y}{3j+i+1}-\lm{Y}{3j+i}\right)^{\otimes 2}-\widehat{A_{3j+i,n}^{\tau}}\right],
\end{align*}
where
\begin{align*}
	\widehat{A_{j,n}^{\tau}}&:=\frac{1}{\Delta_{n}}\left[\frac{1}{\sqrt{m_{n}+m_{n}'}}a\left(X_{j\Delta_{n}},\alpha^{\star}\right)\left(\zeta_{j+1,n}+\zeta_{j+2,n}'\right)+
	\sqrt{\frac{3}{2}}\left(\Lambda_{\star}\right)^{1/2}(\lm{\varepsilon}{j+1}-\lm{\varepsilon}{j})\right]^{\otimes2}
\end{align*}
with the following property
\begin{align*}
	\mathbf{E}_{\theta^{\star}}\left[\widehat{A_{j,n}^{\tau}}|\mathcal{H}_{j}^{n}\right]
	=A\left(X_{j\Delta_{n}},\alpha^{\star}\right)
	+\frac{3}{p_{n}\Delta_{n}}\Lambda_{\star}=A\left(X_{j\Delta_{n}},\alpha^{\star}\right)+3\Delta_{n}^{\frac{2-\tau}{\tau-1}}\Lambda_{\star}=A_{n}^{\tau}\left(X_{j\Delta_{n}},\alpha^{\star},\Lambda_{\star}\right)
\end{align*}
because of Lemma \ref{EvalZeta}, $\Delta_{n}=p_{n}^{1-\tau}$, $\Delta_{n}^{\frac{1}{1-\tau}}=p_{n}$ and $\left(\Delta_{n}p_{n}\right)^{-1}=\Delta_{n}^{\frac{2-\tau}{\tau-1}}$. Furthermore, we have the $L^{p}$-boundedness such that
\begin{align*}
	&\mathbf{E}_{\theta^{\star}}\left[\left\|\widehat{A_{j,n}^{\tau}}-
	\frac{3}{2\Delta_{n}}\left[a\left(X_{j\Delta_{n}},\alpha^{\star}\right)\left(\zeta_{j+1,n}+\zeta_{j+2,n}'\right)+
	\left(\Lambda_{\star}\right)^{1/2}(\lm{\varepsilon}{j+1}-\lm{\varepsilon}{j})\right]^{\otimes2}\right\|^{p}
	\right]^{1/p}\\
	&\le \frac{1}{\Delta_{n}}\left|\frac{3}{2}-\frac{1}{m_{n}+m_{n}'}\right|
	\mathbf{E}_{\theta^{\star}}\left[\left\|
	a\left(X_{j\Delta_{n}},\alpha^{\star}\right)\left(\zeta_{j+1,n}+\zeta_{j+2,n}'\right)\right\|^{2p}
	\right]^{1/p}\\
	&\qquad+\frac{2}{\Delta_{n}}\left|\frac{3}{2}-\sqrt{\frac{3}{2}}\sqrt{\frac{1}{m_{n}+m_{n}'}}\right|
	\mathbf{E}_{\theta^{\star}}\left[\left\|a\left(X_{j\Delta_{n}},\alpha^{\star}\right)\left(\zeta_{j+1,n}+\zeta_{j+2,n}'\right)
	(\lm{\varepsilon}{j+1}-\lm{\varepsilon}{j})^{T}\left(\Lambda_{\star}\right)^{1/2}\right\|^{p}
	\right]^{1/p}\\
	&\le C\left(p\right)\left(\frac{3}{2}-\frac{1}{2/3+1/\left(3p_{n}^2\right)}\right)+\frac{C\left(p\right)}{\Delta_{n}^{1/2}}\left(1-\sqrt{\frac{2/3}{2/3+1/\left(3p_{n}^2\right)}}\right)\\
	&\le \frac{C\left(p\right)}{p_{n}^{2}}
	+\frac{C\left(p\right)}{\Delta_{n}^{1/2}p_{n}^{1/2}}
	\left(1-\sqrt{1-\frac{1}{1+2p_{n}^{2}}}\right)\\
	&\le \frac{C\left(p\right)}{p_{n}^{2}}+\frac{C\left(p\right)}{\Delta_{n}^{1/2}p_{n}^{5/2}}\\
	&\le C\left(p\right)\Delta_{n}^2
\end{align*}
because of $\left\|\zeta_{j+1,n}+\zeta_{j+2,n}\right\|_{p}\le C\left(p\right)\Delta_{n}^{1/2}$ and $\left\|\lm{\varepsilon}{j}\right\|_{p}=C\left(p\right)p_{n}^{-1/2}$ for all $j=0,\ldots,k_{n}-1$ and $n\in\mathbf{N}$, and the Taylor expansion for $f\left(x\right)=\sqrt{1+x}$ around $x=0$. The $L^{p}$-boundedness of $R_{1,n}^{\tau\left(1\right)}$ is led by Lemma \ref{ApproxFunctional} and Burkholder's inequality for martingale, and that of $R_{1,n}^{\tau\left(2\right)}$ can be easily obtained by Lemma \ref{MomentLambda}. With respect to $R_{1,n}^{\tau\left(3\right)}$, we decompose as $R_{1,n}^{\tau\left(3\right)}=\sum_{i=0}^{2}R_{i,1,n}^{\tau\left(3\right)}$ where
\begin{align*}
	&R_{i,1,n}^{\tau\left(3\right)}\\
	&= -\frac{1}{2k_{n}^{1/2}}\sum_{1\le 3j+i\le k_{n}-2}\partial_{\alpha}A_{n}^{\tau}\left(\lm{Y}{3j+i-1},\alpha^{\star},\hat{\Lambda}_{n}\right)^{-1}\left[u_{1},\left(\frac{2\Delta_{n}}{3}\right)^{-1}\left(\lm{Y}{3j+i+1}-\lm{Y}{3j+i}\right)^{\otimes 2}-\widehat{A_{3j+i,n}^{\tau}}\right].
\end{align*}
We only evaluate $R_{0,1,n}^{\tau\left(3\right)}$ and for the case $p$ is an even number. The next inequality holds because of the $L^{p}$-boundedness shown above:
\begin{align*}
	&\mathbf{E}_{\theta^{\star}}\left[\left|R_{0,1,n}^{\tau\left(3\right)}\right|^{p}\right]^{1/p}\\
	&=\mathbf{E}_{\theta^{\star}}\left[\left|\frac{1}{2k_{n}^{1/2}}\sum_{1\le 3j\le k_{n}-2}\partial_{\alpha}A_{n}^{\tau}\left(\lm{Y}{3j-1},\alpha^{\star},\hat{\Lambda}_{n}\right)^{-1}\left[u_{1},\left(\frac{2\Delta_{n}}{3}\right)^{-1}\left(\lm{Y}{3j+1}-\lm{Y}{3j}\right)^{\otimes 2}-\widehat{A_{3j,n}^{\tau}}\right]\right|^{p}\right]^{1/p}\\
	&=\frac{3}{4k_{n}^{1/2}\Delta_{n}}\mathbf{E}_{\theta^{\star}}\left[\left|\sum_{1\le 3j\le k_{n}-2}\partial_{\alpha}A_{n}^{\tau}\left(\lm{Y}{3j-1},\alpha^{\star},\hat{\Lambda}_{n}\right)^{-1}\left[u_{1},\left(\lm{Y}{3j+1}-\lm{Y}{3j}\right)^{\otimes 2}-\frac{2\Delta_{n}}{3}\widehat{A_{3j,n}^{\tau}}\right]\right|^{p}\right]^{1/p}\\
	&\le \frac{3}{4k_{n}^{1/2}\Delta_{n}}\mathbf{E}_{\theta^{\star}}\left[\left|\sum_{1\le 3j\le k_{n}-2}\partial_{\alpha}A_{n}^{\tau}\left(\lm{Y}{3j-1},\alpha^{\star},\hat{\Lambda}_{n}\right)^{-1}\left[u_{1},\left(e_{3j,n}+\Delta_{n}b\left(X_{3j\Delta_{n}}\right)\right)^{\otimes 2}\right]\right|^{p}\right]^{1/p}\\
	&\quad+\frac{3}{2k_{n}^{1/2}\Delta_{n}}\mathbf{E}_{\theta^{\star}}\left[\left|\sum_{1\le 3j\le k_{n}-2}\partial_{\alpha}A_{n}^{\tau}\left(\lm{Y}{3j-1},\alpha^{\star},\hat{\Lambda}_{n}\right)^{-1}\right.\right.\\
	&\hspace{4cm}\left.\left.\left[u_{1},e_{3j,n}\left(a\left(X_{3j\Delta_{n}},\alpha^{\star}\right)\left(\zeta_{3j+1,n}+\zeta_{3j+2,n}'\right)\right)^{T}\right]\right|^{p}\right]^{1/p}\\
	&\quad+\frac{3}{2k_{n}^{1/2}\Delta_{n}}\mathbf{E}_{\theta^{\star}}\left[\left|\sum_{1\le 3j\le k_{n}-2}\partial_{\alpha}A_{n}^{\tau}\left(\lm{Y}{3j-1},\alpha^{\star},\hat{\Lambda}_{n}\right)^{-1}\right.\right.\\
	&\hspace{4cm}\left.\left.\left[u_{1},\Delta_{n}b\left(X_{3j\Delta_{n}}\right)\left(a\left(X_{3j\Delta_{n}},\alpha^{\star}\right)\left(\zeta_{3j+1,n}+\zeta_{3j+2,n}'\right)\right)^{T}\right]\right|^{p}\right]^{1/p}\\
	&\quad+\frac{3}{2k_{n}^{1/2}\Delta_{n}}\mathbf{E}_{\theta^{\star}}\left[\left|\sum_{1\le 3j\le k_{n}-2}\partial_{\alpha}A_{n}^{\tau}\left(\lm{Y}{3j-1},\alpha^{\star},\hat{\Lambda}_{n}\right)^{-1}\right.\right.\\
	&\hspace{4cm}\left.\left.\left[u_{1},\left(e_{3j,n}+\Delta_{n}b\left(X_{3j\Delta_{n}}\right)\right)\left(\left(\Lambda_{\star}\right)^{1/2}(\lm{\varepsilon}{3j+1}-\lm{\varepsilon}{3j})\right)^{T}\right]\right|^{p}\right]^{1/p}\\
	&\quad+o\left(1\right).
\end{align*}
We easily obtain the evaluation for the first term in the right hand side
\begin{align*}
&\frac{3}{4k_{n}^{1/2}\Delta_{n}}\mathbf{E}_{\theta^{\star}}\left[\left|\sum_{1\le 3j\le k_{n}-2}\partial_{\alpha}A_{n}^{\tau}\left(\lm{Y}{3j-1},\alpha^{\star},\hat{\Lambda}_{n}\right)^{-1}\left[u_{1},\left(e_{3j,n}+\Delta_{n}b\left(X_{3j\Delta_{n}}\right)\right)^{\otimes 2}\right]\right|^{p}\right]^{1/p}\\
&\le C\left(p\right)\left|u\right|k_{n}^{1/2}\Delta_{n}\to 0,
\end{align*}
and that for the second term
\begin{align*}
&\frac{3}{2k_{n}^{1/2}\Delta_{n}}\mathbf{E}_{\theta^{\star}}\left[\left|\sum_{1\le 3j\le k_{n}-2}\partial_{\alpha}A_{n}^{\tau}\left(\lm{Y}{3j-1},\alpha^{\star},\hat{\Lambda}_{n}\right)^{-1}\right.\right.\\
&\hspace{3cm}\left.\left.\left[u_{1},e_{3j,n}\left(a\left(X_{3j\Delta_{n}},\alpha^{\star}\right)\left(\zeta_{3j+1,n}+\zeta_{3j+2,n}'\right)\right)^{T}\right]\right|^{p}\right]^{1/p}\\
&=\frac{3}{2k_{n}^{1/2}\Delta_{n}}\mathbf{E}_{\theta^{\star}}\left[\left|\sum_{1\le 3j\le k_{n}-2}\partial_{\alpha}A_{n}^{\tau}\left(\lm{Y}{3j-1},\alpha^{\star},\Lambda_{\star}\right)^{-1}\right.\right.\\
&\hspace{3cm}\left.\left.\left[u_{1},e_{3j,n}\left(a\left(X_{3j\Delta_{n}},\alpha^{\star}\right)\left(\zeta_{3j+1,n}+\zeta_{3j+2,n}'\right)\right)^{T}\right]\right|^{p}\right]^{1/p}+o\left(1\right)\\
&\le C\left(p\right)\left|u\right|k_{n}^{1/2}\Delta_{n}\to 0,
\end{align*}
because of Lemma \ref{ExpansionLM} and Lemma \ref{MomentLambda}. For the third term, we can replace $\hat{\Lambda}_{n}$ with $\Lambda_{\star}$ and $\lm{Y}{3j-1}$ with $X_{3j\Delta_{n}}$ because of Lemma \ref{MomentLambda} and the result from combining Lemma \ref{EvalZeta} and Proposition 12 in \cite{Nakakita-Uchida-2018a}, we denote
\begin{align*}
	\eta_{3j,n}\left(u_{1}\right)
	=\left(a\left(X_{3j\Delta_{n}}\right)\right)^{T}\left(\partial_{\alpha}A_{n}^{\tau}\left(X_{3j\Delta_{n}},
	\alpha^{\star},\Lambda_{\star}\right)\left[u_{1}\right]\right)
	b\left(X_{3j\Delta_{n}}\right)
\end{align*}
which is a $\mathcal{H}_{3j}^{n}$-measurable random variable. Because of Lemma \ref{EvalZeta} and BDG-inequality, we have
\begin{align*}
	&\frac{3}{2k_{n}^{1/2}}\mathbf{E}_{\theta^{\star}}\left[\left|\sum_{1\le 3j\le k_{n}-2}\eta_{3j,n}\left(u_{1}\right)\left[\zeta_{3j+1,n}+\zeta_{3j+2,n}'\right]\right|^{p}\right]^{1/p}\\
	&\le\frac{C\left(p\right)}{k_{n}^{1/2}}\mathbf{E}_{\theta^{\star}}\left[\left(\int_{0}^{k_{n}\Delta_{n}}\sum_{1\le 3j\le k_{n}-2}\left\|\eta_{3j,n}\left(u_{1}\right)\right\|^{2}
	\mathbf{1}_{\left[3j\Delta_{n},\left(3j+1\right)\Delta_{n}\right]}\left(s\right)\mathrm{d}s\right)^{p/2}\right]^{1/p}\\
	&\le \frac{C\left(p\right)}{k_{n}^{1/2}}\mathbf{E}_{\theta^{\star}}\left[\left(\int_{0}^{k_{n}\Delta_{n}}\sum_{1\le 3j\le k_{n}-2}\left\|\eta_{3j,n}\left(u_{1}\right)\right\|^{p}
	\mathbf{1}_{\left[3j\Delta_{n},\left(3j+1\right)\Delta_{n}\right]}\left(s\right)\mathrm{d}s\right)\left(\int_{0}^{k_{n}\Delta_{n}}\mathrm{d}s\right)^{p/2-1}\right]^{1/p}\\
	&= \frac{C\left(p\right)\left(k_{n}\Delta_{n}\right)^{1/2-1/p}}{k_{n}^{1/2}}\left(\int_{0}^{k_{n}\Delta_{n}}\sum_{1\le 3j\le k_{n}-2}\mathbf{E}_{\theta^{\star}}\left[\left\|\eta_{3j,n}\left(u_{1}\right)\right\|^{p}
	\right]\mathbf{1}_{\left[3j\Delta_{n},\left(3j+1\right)\Delta_{n}\right]}\left(s\right)\mathrm{d}s\right)^{1/p}\\
	&\le \frac{C\left(p\right)\left(k_{n}\Delta_{n}\right)^{1/2}}{k_{n}^{1/2}}\left|u\right|\\
	&\le C\left(p\right)\Delta_{n}^{1/2}\\
	&\to 0.
\end{align*}
It is obvious that the fourth term can be evaluated as bounded because $\left\{\varepsilon_{ih_{n}}\right\}$ is independent of $X$ and i.i.d. Therefore, we obtain $\left\|R_{0,1,n}^{\tau\left(3\right)}\right\|_{p}<\infty $ and $\left\|R_{1,n}^{\tau\left(3\right)}\right\|_{p}<\infty$.

With respect to $M_{1,n}^{\tau}$, we utilise Burkholder's inequality for martingale: let us define $M_{i,1,n}^{\tau}$ for $i=0,1,2$ as same as $R_{i,1,n}^{\tau\left(3\right)}$ and then
\begin{align*}
	&\mathbf{E}_{\theta^{\star}}\left[\left|M_{0,1,n}^{\tau}\right|^{p}\right]\\
	&\le C\left(p\right)\mathbf{E}_{\theta^{\star}}\left[\left|\frac{1}{4k_{n}}\sum_{1\le 3j\le k_{n}-2}\left|
	\partial_{\alpha}A_{n}^{\tau}\left(X_{j\Delta_{n}},\alpha^{\star},\Lambda_{\star}\right)^{-1}\left[u_{1},\widehat{A_{j,n}^{\tau}}\right]\right.\right.\right.\\
	&\hspace{4cm}\left.\left.\left.+\partial_{\alpha}\log\det A_{n}^{\tau}\left(X_{j\Delta_{n}},\alpha^{\star},\Lambda_{\star}\right)\left[u_{1}\right]\right|^{2}\right|^{p/2}\right]\\
	&\le \frac{C\left(p\right)}{k_{n}}\sum_{1\le 3j\le k_{n}-2}\mathbf{E}_{\theta^{\star}}\left[\left|
	\partial_{\alpha}A_{n}^{\tau}\left(X_{j\Delta_{n}},\alpha^{\star},\Lambda_{\star}\right)^{-1}\left[u_{1},\widehat{A_{j,n}^{\tau}}\right]+\partial_{\alpha}\log\det A_{n}^{\tau}\left(X_{j\Delta_{n}},\alpha^{\star},\Lambda_{\star}\right)\left[u_{1}\right]\right|^{p}\right]\\
	&<\infty
\end{align*}
because of the integrability.

In the next place, we give the proof for (b).
Let us denote
\begin{align*}
\mathbb{Y}_{1,n}^{\tau\left(\dagger\right)}\left(\alpha;\vartheta^{\star}\right)	&=-\frac{1}{2k_{n}}\sum_{j=1}^{k_{n}-2}
\left(\left(A_{n}^{\tau}\left(\lm{Y}{j-1},\alpha,\hat{\Lambda}_{n}\right)^{-1}-A_{n}^{\tau}\left(\lm{Y}{j-1},\alpha^{\star},\hat{\Lambda}_{n}\right)^{-1}\right)\left[A_{n}^{\tau}\left(X_{j\Delta_{n}},\alpha^{\star},\Lambda_{\star}\right)\right]\right.\\
&\hspace{3cm}\left.+\log\frac{\det A_{n}^{\tau}\left(\lm{Y}{j-1},\alpha,\hat{\Lambda}_{n}\right)}{\det A_{n}^{\tau}\left(\lm{Y}{j-1},\alpha^{\star},\hat{\Lambda}_{n}\right)}\right).
\end{align*}
Define $R_{1,n}^{\tau\left(\dagger\right)}$ by
\begin{align*}
	R_{1,n}^{\tau\left(\dagger\right)}=\mathbb{Y}_{1,n}^{\tau}\left(\alpha;\vartheta^{\star}\right)-\mathbb{Y}_{1,n}^{\tau\left(\dagger\right)}\left(\alpha;\vartheta^{\star}\right)-M_{1,n}^{\tau\left(\dagger\right)}
\end{align*}
for
\begin{align*}
M_{1,n}^{\tau\left(\dagger\right)}&=-\frac{1}{2k_{n}}\sum_{j=1}^{k_{n}-2}
\left(\left(A_{n}^{\tau}\left(\lm{Y}{j-1},\alpha,\hat{\Lambda}_{n}\right)^{-1}-A_{n}^{\tau}\left(\lm{Y}{j-1},\alpha^{\star},\hat{\Lambda}_{n}\right)^{-1}\right)\left[\widehat{A_{j,n}^{\tau}}-A_{n}^{\tau}\left(X_{j\Delta_{n}},\alpha^{\star},\Lambda_{\star}\right)\right]\right).
\end{align*}
Firstly we show $L^p$-boundedness of $k_{n}^{\epsilon_{1}}R_{1,n}^{\tau\left(\dagger\right)}$ uniformly for $n$ and $\alpha$ for every $p$. We have the representation such that
\begin{align*}
	&R_{1,n}^{\tau\left(\dagger\right)}\\
	&=-\frac{1}{2k_{n}}\sum_{j=1}^{k_{n}-2}
	\left(\left(A_{n}^{\tau}\left(\lm{Y}{j-1},\alpha,\hat{\Lambda}_{n}\right)^{-1}-A_{n}^{\tau}\left(\lm{Y}{j-1},\alpha^{\star},\hat{\Lambda}_{n}\right)^{-1}\right)\left[\left(\lm{Y}{j+1}-\lm{Y}{j}\right)^{\otimes 2}\right]\left(\frac{2}{3}\Delta_{n}\right)^{-1}\right.\\
	&\hspace{3cm}\left.+\log\frac{\det A_{n}^{\tau}\left(\lm{Y}{j-1},\alpha,\hat{\Lambda}_{n}\right)}{\det A_{n}^{\tau}\left(\lm{Y}{j-1},\alpha^{\star},\hat{\Lambda}_{n}\right)}\right)\\
	&\qquad+\frac{1}{2k_{n}}\sum_{j=1}^{k_{n}-2}
	\left(\left(A_{n}^{\tau}\left(\lm{Y}{j-1},\alpha,\hat{\Lambda}_{n}\right)^{-1}-A_{n}^{\tau}\left(\lm{Y}{j-1},\alpha^{\star},\hat{\Lambda}_{n}\right)^{-1}\right)\left[A_{n}^{\tau}\left(X_{j\Delta_{n}},\alpha^{\star},\Lambda_{\star}\right)\right]\right.\\
	&\hspace{3cm}\left.+\log\frac{\det A_{n}^{\tau}\left(\lm{Y}{j-1},\alpha,\hat{\Lambda}_{n}\right)}{\det A_{n}^{\tau}\left(\lm{Y}{j-1},\alpha^{\star},\hat{\Lambda}_{n}\right)}\right)\\
	&\qquad+\frac{1}{2k_{n}}\sum_{j=1}^{k_{n}-2}
	\left(\left(A_{n}^{\tau}\left(\lm{Y}{j-1},\alpha,\hat{\Lambda}_{n}\right)^{-1}-A_{n}^{\tau}\left(\lm{Y}{j-1},\alpha^{\star},\hat{\Lambda}_{n}\right)^{-1}\right)\left[\widehat{A_{j,n}^{\tau}}-A_{n}^{\tau}\left(X_{j\Delta_{n}},\alpha^{\star},\Lambda_{\star}\right)\right]\right)\\
	&=-\frac{1}{2k_{n}}\sum_{j=1}^{k_{n}-2}
	\left(\left(A_{n}^{\tau}\left(\lm{Y}{j-1},\alpha,\hat{\Lambda}_{n}\right)^{-1}-A_{n}^{\tau}\left(\lm{Y}{j-1},\alpha^{\star},\hat{\Lambda}_{n}\right)^{-1}\right)\left[\left(\lm{Y}{j+1}-\lm{Y}{j}\right)^{\otimes 2}\right]\left(\frac{2}{3}\Delta_{n}\right)^{-1}\right)\\
	&\qquad+\frac{1}{2k_{n}}\sum_{j=1}^{k_{n}-2}
	\left(\left(A_{n}^{\tau}\left(\lm{Y}{j-1},\alpha,\hat{\Lambda}_{n}\right)^{-1}-A_{n}^{\tau}\left(\lm{Y}{j-1},\alpha^{\star},\hat{\Lambda}_{n}\right)^{-1}\right)\left[A_{n}^{\tau}\left(X_{j\Delta_{n}},\alpha^{\star},\Lambda_{\star}\right)\right]\right)\\
	&\qquad+\frac{1}{2k_{n}}\sum_{j=1}^{k_{n}-2}
	\left(\left(A_{n}^{\tau}\left(\lm{Y}{j-1},\alpha,\hat{\Lambda}_{n}\right)^{-1}-A_{n}^{\tau}\left(\lm{Y}{j-1},\alpha^{\star},\hat{\Lambda}_{n}\right)^{-1}\right)\left[\widehat{A_{j,n}^{\tau}}-A_{n}^{\tau}\left(X_{j\Delta_{n}},\alpha^{\star},\Lambda_{\star}\right)\right]\right)\\
	&=-\frac{1}{2k_{n}}\sum_{j=1}^{k_{n}-2}
	\left(A_{n}^{\tau}\left(\lm{Y}{j-1},\alpha,\hat{\Lambda}_{n}\right)^{-1}-A_{n}^{\tau}\left(\lm{Y}{j-1},\alpha^{\star},\hat{\Lambda}_{n}\right)^{-1}\right)\left[\left(\frac{2}{3}\Delta_{n}\right)^{-1}\left(\lm{Y}{j+1}-\lm{Y}{j}\right)^{\otimes 2}-\widehat{A_{j,n}^{\tau}}\right].
\end{align*}
Because of Lemma \ref{ExpansionLM}, the following evaluation holds:
\begin{align*}
	&\left\|\left(\frac{2}{3}\Delta_{n}\right)^{-1}\left(\lm{Y}{j+1}-\lm{Y}{j}\right)^{\otimes 2}-\widehat{A_{j,n}^{\tau}}\right\|_{p}\\
	&\le \left\|\left(\frac{2}{3}\Delta_{n}\right)^{-1}\left[\left(\lm{Y}{j+1}-\lm{Y}{j}\right)^{\otimes 2}-\left(a\left(X_{j\Delta_{n}},\alpha^{\star}\right)\left(\zeta_{j+1,n}+\zeta_{j+2,n}'\right)+\left(\Lambda_{\star}\right)^{1/2}(\lm{\varepsilon}{j+1}-\lm{\varepsilon}{j})\right)^{\otimes2}\right]\right\|_{p}\\
	&\qquad +C\left(p\right)\Delta_{n}\\
	&=\left(\frac{2}{3}\Delta_{n}\right)^{-1}\left\|e_{j,n}^{\otimes2}-e_{j,n}\left(a\left(X_{j\Delta_{n}},\alpha^{\star}\right)\left(\zeta_{j+1,n}+\zeta_{j+2,n}'\right)+\left(\Lambda_{\star}\right)^{1/2}(\lm{\varepsilon}{j+1}-\lm{\varepsilon}{j})\right)^{T}\right.\\
	&\hspace{4cm} \left.- \left(a\left(X_{j\Delta_{n}},\alpha^{\star}\right)\left(\zeta_{j+1,n}+\zeta_{j+2,n}'\right)+\left(\Lambda_{\star}\right)^{1/2}(\lm{\varepsilon}{j+1}-\lm{\varepsilon}{j})\right)
	e_{j,n}^{T}\right\|_{p}\\
	&\qquad +C\left(p\right)\Delta_{n}\\
	&\le \left(\frac{2}{3}\Delta_{n}\right)^{-1}\left(\left\|e_{j,n}\right\|_{p}^{2}+2\left\|a\left(X_{j\Delta_{n}},\alpha^{\star}\right)\left(\zeta_{j+1,n}+\zeta_{j+2,n}'\right)+\left(\Lambda_{\star}\right)^{1/2}(\lm{\varepsilon}{j+1}-\lm{\varepsilon}{j})\right\|_{2p}
	\left\|e_{j,n}\right\|_{2p}\right)\\
	&\qquad +C\left(p\right)\Delta_{n}\\
	&\le C\left(p\right)\left(\Delta_{n}+\Delta_{n}^{1/2}\right).
\end{align*}
Hence, we have the evaluation
\begin{align*}
	\sup_{\alpha\in\Theta_{1}}\sup_{n\in\mathbf{N}}\left\|R_{1,n}^{\tau\left(\dagger\right)}\right\|_{p}
	\le C\left(p\right)\Delta_{n}+C\left(p\right)\Delta_{n}^{1/2}\le C\Delta_{n}^{1/2},
\end{align*}
and hence
\begin{align*}
	\sup_{\alpha\in\Theta_{1}}\sup_{n\in\mathbf{N}}\left\|k_{n}^{\epsilon_{1}}R_{1,n}^{\tau\left(\dagger\right)}\right\|_{p}
	\le C\left(p\right)k_{n}^{\epsilon_{1}}\Delta_{n}^{1/2} = C\left(p\right)\left(k_{n}^{\epsilon_{0}}\Delta_{n}\right)^{1/2}\to 0.
\end{align*}
In the next place, we see the same uniform $L^{p}$-boundedness of $k_{n}^{\epsilon_{1}}M_{1,n}^{\tau\left(\dagger\right)}$ for every $p$. As the approximation, we set $M_{1,n}^{\tau\left(\ddagger\right)}:=\sum_{i=0}^{2}M_{i,1,n}^{\tau\left(\ddagger\right)}$ where for $i=0,1,2$,
\begin{align*}
M_{i,1,n}^{\tau\left(\ddagger\right)}&:=-\frac{1}{2k_{n}}\sum_{1\le 3j+i\le k_{n}-2}\mu_{3j+i,n},
\end{align*}
where
\begin{align*}
	\mu_{3j+i,n}=\left(A_{n}^{\tau}\left(\lm{Y}{3j+i-1},\alpha,\Lambda_{\star}\right)^{-1}-A_{n}^{\tau}\left(\lm{Y}{3j+i-1},\alpha^{\star},\Lambda_{\star}\right)^{-1}\right)
	\left[\widehat{A_{3j+i,n}^{\tau}}-A_{n}^{\tau}\left(X_{\left(3j+i\right)\Delta_{n}},\alpha^{\star},\Lambda_{\star}\right)\right].
\end{align*}
It is easy to show $\mathbf{E}\left[\sup_{\alpha\in\Theta_{1}}k_{n}^{\epsilon_{1}}\left|M_{1,n}^{\tau\left(\dagger\right)}-M_{1,n}^{\tau\left(\ddagger\right)}\right|^{p}\right]^{1/p}\le C\left(p\right)k_{n}^{\epsilon_{1}}n^{-1/2}\le C\left(p\right)h_{n}^{1/2}\to0$ for Lemma \ref{MomentLambda}. For simplicity, we only evaluate $k_{n}^{\epsilon_{1}}M_{0,1,n}^{\tau\left(\ddagger\right)}$. We have for all $p$,
\begin{align*}
&\mathbf{E}_{\theta^{\star}}\left[\left|\mu_{3j,n}\right|^{p}\right]\\
&=\mathbf{E}_{\theta^{\star}}\left[\left|\left(A_{n}^{\tau}\left(\lm{Y}{3j-1},\alpha,\Lambda_{\star}\right)^{-1}-A_{n}^{\tau}\left(\lm{Y}{3j-1},\alpha^{\star},\Lambda_{\star}\right)^{-1}\right)\left[\widehat{A_{3j,n}^{\tau}}-A_{n}^{\tau}\left(X_{3j\Delta_{n}},\alpha^{\star},\Lambda_{\star}\right)\right]\right|^{p}\right]\\
&\le \mathbf{E}_{\theta^{\star}}\left[\left\|A_{n}^{\tau}\left(\lm{Y}{3j-1},\alpha,\Lambda_{\star}\right)^{-1}-A_{n}^{\tau}\left(\lm{Y}{3j-1},\alpha^{\star},\Lambda_{\star}\right)^{-1}\right\|^{p}\left\|\widehat{A_{3j,n}^{\tau}}-A_{n}^{\tau}\left(X_{3j\Delta_{n}},\alpha^{\star},\Lambda_{\star}\right)\right\|^{p}\right]\\
&\le C\left(p\right)\mathbf{E}_{\theta^{\star}}\left[\left\|\widehat{A_{3j,n}^{\tau}}-A_{n}^{\tau}\left(X_{3j\Delta_{n}},\alpha^{\star},\Lambda_{\star}\right)\right\|^{2p}\right]^{1/2}\\
&\le C\left(p\right).
\end{align*}
Hence by Burkholder's inequality, for all $p$,
\begin{align*}
	\mathbf{E}_{\theta^{\star}}\left[\left|k_{n}^{\epsilon_{1}}M_{0,1,n}^{\tau\left(\ddagger\right)}\right|^{p}\right]
	&\le C\left(p\right)k_{n}^{\epsilon_{1}p}\mathbf{E}_{\theta^{\star}}\left[\left|\frac{1}{k_{n}^{2}}\sum_{1\le 3j\le k_{n}-2}\mu_{3j,n}^{2}\right|^{p/2}\right]\\
	&\le C\left(p\right)k_{n}^{\epsilon_{1}p}k_{n}^{-p/2}\frac{1}{k_{n}}\sum_{1\le 3j\le k_{n}-2}\mathbf{E}_{\theta^{\star}}\left[\left|\mu_{3j,n}^{2}\right|^{p/2}\right]\\
	&\le C\left(p\right)k_{n}^{\left(\epsilon_{1}-1/2\right)p}\frac{1}{k_{n}}\sum_{1\le 3j\le k_{n}-2}\mathbf{E}_{\theta^{\star}}\left[\left|\mu_{3j,n}\right|^{p}\right]\\
	&\le C\left(p\right)k_{n}^{\left(\epsilon_{1}-1/2\right)p}
\end{align*}
and then $\sup_{n,\theta^{\star}}\left\|k_{n}^{\epsilon_{1}}M_{1,n}^{\tau\left(\ddagger\right)}\right\|_{p}<\infty$. With the same procedure, we obtain the uniform $L^p$-boundedness of $k_{n}^{\epsilon_{1}}\partial_{\alpha}R_{1,n}^{\tau\left(\dagger\right)}$ and $k_{n}^{\epsilon_{1}}\partial_{\alpha}M_{1,n}^{\tau\left(\ddagger\right)}$. Sobolev's inequality leads to
\begin{align*}
	\sup_{n\in\mathbf{N}}\left\|\sup_{\alpha\in\Theta_{1}}\left|k_{n}^{\epsilon_{1}}R_{1,n}^{\tau\left(\dagger\right)}\right|\right\|_{p}<\infty,\ \sup_{n\in\mathbf{N}}\left\|\sup_{\alpha\in\Theta_{1}}\left|k_{n}^{\epsilon_{1}}M_{1,n}^{\tau\left(\ddagger\right)}\right|\right\|_{p}<\infty
\end{align*}
and then $\sup_{n\in\mathbf{N}}\left\|\sup_{\alpha\in\Theta_{1}}\left|k_{n}^{\epsilon_{1}}M_{1,n}^{\tau\left(\dagger\right)}\right|\right\|_{p}<\infty$.
Note that for
\begin{align*}
	\mathbb{Y}_{1,n}^{\tau\left(\ddagger\right)}\left(\alpha;\vartheta^{\star}\right)	&=-\frac{1}{2k_{n}}\sum_{j=1}^{k_{n}-2}
	\left(\left(A_{n}^{\tau}\left(X_{j\Delta_{n}},\alpha,\Lambda_{\star}\right)^{-1}-A_{n}^{\tau}\left(X_{j\Delta_{n}},\alpha^{\star},\Lambda_{\star}\right)^{-1}\right)\left[A_{n}^{\tau}\left(X_{j\Delta_{n}},\alpha^{\star},\Lambda_{\star}\right)\right]\right.\\
	&\hspace{3cm}\left.+\log\frac{\det A_{n}^{\tau}\left(X_{j\Delta_{n}},\alpha,\Lambda_{\star}\right)}{\det A_{n}^{\tau}\left(X_{j\Delta_{n}},\alpha^{\star},\Lambda_{\star}\right)}\right),
\end{align*}
we can evaluate $\sup_{n\in\mathbf{N}}\left\|\sup_{\alpha\in\Theta_{1}}\left|k_{n}^{\epsilon_{1}}\left(\mathbb{Y}_{1,n}^{\tau\left(\ddagger\right)}\left(\alpha;\vartheta^{\star}\right)\left(\alpha;\vartheta^{\star}\right)-\mathbb{Y}_{1,n}^{\tau\left(\dagger\right)}\left(\alpha;\vartheta^{\star}\right)\left(\alpha;\vartheta^{\star}\right)\right)\right|\right\|_{p}<\infty$ because of Lemma \ref{ApproxSumLM} and Lemma \ref{EvalFuncLambda}. Hence the discussion of Remark \ref{RemarkCentred} leads to the proof.
\end{proof}

\begin{lemma}\label{alphaInfo}
	\begin{enumerate}
	\item[(a)]For any $M_{3}>0$,
	\begin{align*}
		\sup_{n\in\mathbf{N}}\mathbf{E}_{\theta^{\star}}\left[\left(k_{n}^{-1}\sup_{\vartheta\in\Xi}
		\left|\partial_{\alpha}^{3}\mathbb{H}_{1,n}^{\tau}\left(\alpha;\Lambda\right)\right|\right)^{M_{3}}\right]<\infty.
	\end{align*}
	
	\item[(b)]Let $\epsilon_{1}=\epsilon_{0}/2$. Then for $M_{4}>0$,
	\begin{align*}
	\sup_{n\in\mathbf{N}}\mathbf{E}_{\theta^{\star}}\left[\left(k_{n}^{\epsilon_{1}}\left|\Gamma_{1,n}^{\tau}\left(\alpha^{\star};\vartheta^{\star}\right)-\Gamma_{1}^{\tau}\left(\vartheta^{\star}\right)\right|\right)^{M_{4}}\right]<\infty.
	\end{align*}
	\end{enumerate}
\end{lemma}

\begin{proof}With respect to (a), we have
\begin{align*}
	&\sup_{\vartheta\in\Xi}\left|\partial_{\alpha}^{3}\mathbb{H}_{1,n}^{\tau}\left(\alpha;\Lambda\right)\right|\\
	&=\sup_{\vartheta\in\Xi}\left|\frac{1}{2}\partial_{\alpha}^{3}\sum_{j=1}^{k_{n}-2}
	\left(\left(\frac{2}{3}\Delta_{n}A_{n}^{\tau}\left(\lm{Y}{j-1},\alpha,\Lambda\right)\right)^{-1}\left[\left(\lm{Y}{j+1}-\lm{Y}{j}\right)^{\otimes 2}\right]+\log\det A_{n}^{\tau}\left(\lm{Y}{j-1},\alpha,\Lambda\right)\right)\right|\\
	&\le \sup_{\vartheta\in\Xi}
	\left|\sum_{j=1}^{k_{n}-2}\partial_{\alpha}^{3}\left(A_{n}^{\tau}\left(\lm{Y}{j-1},\alpha,\Lambda\right)\right)^{-1}
	\left[\frac{3}{4\Delta_{n}}\left(\lm{Y}{j+1}-\lm{Y}{j}\right)^{\otimes 2}\right]
	\right|\\
	&\qquad+\sup_{\vartheta\in\Xi}\frac{1}{2}\sum_{j=1}^{k_{n}-2}
	\left|\partial_{\alpha}^{3}\log\det A_{n}^{\tau}\left(\lm{Y}{j-1},\alpha,\Lambda\right)\right|\\
	&\le C\sum_{j=1}^{k_{n}-2}\left(1+\left|\lm{Y}{j-1}\right|\right)^{C}
	\Delta_{n}^{-1}\left|\lm{Y}{j+1}-\lm{Y}{j}\right|^{2}
	+C\sum_{j=1}^{k_{n}-2}\left(1+\left|\lm{Y}{j-1}\right|\right)^{C}
\end{align*}
and hence
\begin{align*}
	&\sup_{n\in\mathbf{N}}\mathbf{E}_{\theta^{\star}}\left[\left(k_{n}^{-1}\sup_{\vartheta\in\Xi}
	\left|\partial_{\alpha}^{3}\mathbb{H}_{1,n}^{\tau}\left(\alpha;\Lambda\right)\right|\right)^{M_{3}}\right]\\
	&\le C\sup_{n\in\mathbf{N}}\mathbf{E}_{\theta^{\star}}\left[\left(k_{n}^{-1}\sum_{j=1}^{k_{n}-2}\left(1+\left|\lm{Y}{j-1}\right|\right)^{C}
	\Delta_{n}^{-1}\left|\lm{Y}{j+1}-\lm{Y}{j}\right|^{2}\right)^{M_{3}}\right]\\
	&\qquad+C\sup_{n\in\mathbf{N}}\mathbf{E}_{\theta^{\star}}\left[\left(k_{n}^{-1}\sum_{j=1}^{k_{n}-2}\left(1+\left|\lm{Y}{j-1}\right|\right)^{C}
	\right)^{M_{3}}\right]\\
	&\le C\sup_{n\in\mathbf{N}}k_{n}^{-1}\sum_{j=1}^{k_{n}-2}\mathbf{E}_{\theta^{\star}}\left[1+\left|X_{j\Delta_{n}}\right|^{C}
	\right]\\
	&<\infty.
\end{align*}

For (b), the discussion same as Lemma \ref{alphaScore} leads to the result.
\end{proof}

\begin{proposition}\label{MomentAlpha}
	For any $p>0$, 
	\begin{align*}
	\sup_{n\in\mathbf{N}}\mathbf{E}_{\theta^{\star}}\left[\left|\sqrt{k_{n}}\left(\hat{\alpha}_{n}-\alpha^{\star}\right)\right|^{p}\right]<\infty,\ \sup_{n\in\mathbf{N}}\mathbf{E}_{\theta^{\star}}\left[\left|\sqrt{k_{n}}\left(\tilde{\alpha}_{n}-\alpha^{\star}\right)\right|^{p}\right]<\infty.
	\end{align*}
\end{proposition}

\begin{proof}
Theorem 3 in \cite{Yoshida-2011}, Lemma \ref{alphaScore} and Lemma \ref{alphaInfo} lead to the following polynomial large deviation inequality
\begin{align*}
	P_{\theta^{\star}}\left[\sup_{u_{1}\in V_{1,n}^{\tau}\left(r,\alpha^{\star}\right)}\mathbb{Z}_{1,n}^{\tau}\left(u_{1};\hat{\Lambda}_{n},\alpha^{\star}\right)\ge e^{-r}\right] \le \frac{C\left(L\right)}{r^{L}}
\end{align*}
for all $r>0$ and $n\in\mathbf{N}$. 
The $L^{p}$-boundedness of $\sqrt{k_{n}}\left(\hat{\alpha}_{n}-\alpha^{\star}\right)$ is then obtained with the discussion parallel to \cite{Yoshida-2011}.

With respect to the Bayes-type estimator, we need to verify the next boundedness: there exists $\delta_{1}>0$ and $C>0$ such that
\begin{align*}
\sup_{n\in\mathbf{N}}\mathbf{E}_{\theta^{\star}}\left[\left(\int_{u_{1}:\left|u_{1}\right|\le \delta_{1}}\mathbb{Z}_{1,n}^{\tau}\left(u_{1};\hat{\Lambda}_{n},\alpha^{\star}\right)\mathrm{d}u_{1}\right)^{-1}\right]<\infty.
\end{align*}
Because of the Lemma 2 in \cite{Yoshida-2011}, it is sufficient to show that for some $p>d$, $\delta>0$ and $C>0$,
\begin{align*}
\sup_{n\in\mathbf{N}}\mathbf{E}_{\theta^{\star}}\left[\left|\log\mathbb{Z}_{1,n}^{\tau}\left(u_{1};\hat{\Lambda}_{n},\alpha^{\star}\right)\right|^{p}\right]\le C\left|u_{1}\right|^{p}\quad^\forall u_{1}\text{ s.t. }\left|u_{1}\right|\le \delta
\end{align*}
and actually it is easy to obtain by Lemma \ref{alphaScore} and Lemma \ref{alphaInfo}.
\end{proof}

\begin{lemma}\label{betaScore}
\begin{enumerate}
	\item[{(a)}] For every $p>0$, 
	\begin{align*}
	\sup_{n\in\mathbf{N}}\mathbf{E}_{\theta^{\star}}\left[\left|\Delta_{2,n}^{\mathrm{ML}}\left(\vartheta^{\star}\right)\right|^{p}\right]<\infty,\ \sup_{n\in\mathbf{N}}\mathbf{E}_{\theta^{\star}}\left[\left|\Delta_{2,n}^{\mathrm{Bayes}}\left(\vartheta^{\star}\right)\right|^{p}\right]<\infty.
	\end{align*}
	\item[{(b)}] Let $\epsilon_{1}=\epsilon_{0}/2$. Then for every $p>0$,
	\begin{align*}
	\sup_{n\in\mathbf{N}}\left\|\sup_{\beta\in\Theta_{2}}\left(k_{n}\Delta_{n}\right)^{\epsilon_{1}}\left|\mathbb{Y}_{2,n}^{\mathrm{ML}}\left(\beta;\vartheta^{\star}\right)-\mathbb{Y}_{2}\left(\beta;\vartheta^{\star}\right)\right|\right\|_{p}&<\infty,\\
	\sup_{n\in\mathbf{N}}\left\|\sup_{\beta\in\Theta_{2}}\left(k_{n}\Delta_{n}\right)^{\epsilon_{1}}\left|\mathbb{Y}_{2,n}^{\mathrm{Bayes}}\left(\beta;\vartheta^{\star}\right)-\mathbb{Y}_{2}\left(\beta;\vartheta^{\star}\right)\right|\right\|_{p}&<\infty.
	\end{align*}
\end{enumerate}
\end{lemma}

\begin{proof}
We only show the proof for $\Delta_{2,n}^{\mathrm{ML}}$ and $\mathbb{Y}_{2,n}^{\mathrm{ML}}$ since the proof for $\Delta_{2,n}^{\mathrm{Bayes}}$ and $\mathbb{Y}_{2,n}^{\mathrm{Bayes}}$ are quite parallel. For (a), we decompose
\begin{align*}
	\Delta_{2,n}^{\mathrm{ML}}\left(\vartheta^{\star}\right)\left[u_{2}\right]
	&=M_{2,n}^{\mathrm{ML}}+R_{2,n}^{\mathrm{ML}},
\end{align*}
where
\begin{align*}
	M_{2,n}^{\mathrm{ML}}
	&=\frac{1}{\left(k_{n}\Delta_{n}\right)^{1/2}}\sum_{j=1}^{k_{n}-2}A\left(\lm{Y}{j-1},\hat{\alpha}_{n}\right)^{-1}\left[\partial_{\beta}b\left(\lm{Y}{j-1},\beta^{\star}\right)u_{2},a\left(X_{j\Delta_{n}}\right)\left(\zeta_{j+1,n}+\zeta_{j+2,n}'\right)\right]\\
	&\qquad+\frac{1}{\left(k_{n}\Delta_{n}\right)^{1/2}}\sum_{j=1}^{k_{n}-2}A\left(\lm{Y}{j-1},\hat{\alpha}_{n}\right)^{-1}\left[\partial_{\beta}b\left(\lm{Y}{j-1},\beta^{\star}\right)u_{2},\left(\Lambda_{\star}\right)^{1/2}\left(\lm{\varepsilon}{j+1}-\lm{\varepsilon}{j}\right)\right],\\
	R_{2,n}^{\mathrm{ML}}&=\frac{\Delta_{n}}{\left(k_{n}\Delta_{n}\right)^{1/2}}\sum_{j=1}^{k_{n}-2}A\left(\lm{Y}{j-1},\hat{\alpha}_{n}\right)^{-1}\left[\partial_{\beta}b\left(\lm{Y}{j-1},\beta^{\star}\right)u_{2},b\left(X_{j\Delta_{n}}\right)-b\left(\lm{Y}{j-1}\right)\right]\\
	&\qquad+\frac{1}{\left(k_{n}\Delta_{n}\right)^{1/2}}\sum_{j=1}^{k_{n}-2}A\left(\lm{Y}{j-1},\hat{\alpha}_{n}\right)^{-1}\left[\partial_{\beta}b\left(\lm{Y}{j-1},\beta^{\star}\right)u_{2},e_{j,n}\right].
\end{align*}
We can use $L^{p}$-boundedness of $\sqrt{k_{n}}\left(\hat{\alpha}_{n}-\alpha^{\star}\right)$, and Burkholder's inequality; then we obtain have
\begin{align*}
	\sup_{n\in\mathbf{N}}\mathbf{E}_{\theta^{\star}}\left[\left|M_{2,n}^{\mathrm{ML}}\right|^{p}\right]^{1/p}
	&\le C\left(p\right),
\end{align*}
and for the residuals, Lemma \ref{ApproxSumLM} and Lemma \ref{ExpansionLM} lead to
\begin{align*}
	\mathbf{E}_{\theta^{\star}}\left[\left|R_{2,n}^{\mathrm{ML}}\right|^{p}\right]^{1/p}
	&\le C\left(p\right)\sqrt{k_{n}}\Delta_{n}\to0.
\end{align*}
Then we obtain (a).
We prove (b) in the second place. We decompose $\mathbb{Y}_{2,n}^{\mathrm{ML}}\left(\beta;\vartheta^{\star}\right)$ as 
\begin{align*}
	\mathbb{Y}_{2,n}^{\mathrm{ML}}\left(\beta;\vartheta^{\star}\right)
	=M_{2,n}^{\mathrm{ML}\left(\dagger\right)}\left(\hat{\alpha}_{n},\beta\right)
	+R_{2,n}^{\mathrm{ML}\left(\dagger\right)}\left(\hat{\alpha}_{n},\beta\right)
	+\mathbb{Y}_{2,n}^{\mathrm{ML}\left(\dagger\right)}\left(\beta;\vartheta^{\star}\right),
\end{align*}
where
\begin{align*}
	M_{2,n}^{\mathrm{ML}\left(\dagger\right)}\left(\alpha,\beta\right)&=\frac{1}{k_{n}\Delta_{n}}\sum_{j=1}^{k_{n}-2}A\left(\lm{Y}{j-1},\alpha\right)^{-1}\left[b\left(\lm{Y}{j-1},\beta\right),
	a\left(X_{j\Delta_{n}}\right)\left(\zeta_{j+1,n}+\zeta_{j+2,n}\right)\right]\\
	&\qquad-\frac{1}{k_{n}\Delta_{n}}\sum_{j=1}^{k_{n}-2}A\left(\lm{Y}{j-1},\alpha\right)^{-1}\left[b\left(\lm{Y}{j-1},\beta^{\star}\right),
	a\left(X_{j\Delta_{n}}\right)\left(\zeta_{j+1,n}+\zeta_{j+2,n}\right)\right],\\
	&\qquad+\frac{1}{k_{n}\Delta_{n}}\sum_{j=1}^{k_{n}-2}A\left(\lm{Y}{j-1},\alpha\right)^{-1}\left[b\left(\lm{Y}{j-1},\beta\right),
	\left(\Lambda_{\star}\right)^{1/2}\left(\lm{\varepsilon}{j+1}-\lm{\varepsilon}{j}\right)\right]\\
	&\qquad-\frac{1}{k_{n}\Delta_{n}}\sum_{j=1}^{k_{n}-2}A\left(\lm{Y}{j-1},\alpha\right)^{-1}\left[b\left(\lm{Y}{j-1},\beta^{\star}\right),
	\left(\Lambda_{\star}\right)^{1/2}\left(\lm{\varepsilon}{j+1}-\lm{\varepsilon}{j}\right)\right],\\
	R_{2,n}^{\mathrm{ML}\left(\dagger\right)}\left(\alpha,\beta\right)&=\frac{1}{k_{n}\Delta_{n}}\sum_{j=1}^{k_{n}-2}A\left(\lm{Y}{j-1},\alpha\right)^{-1}\left[b\left(\lm{Y}{j-1},\beta\right),e_{j,n}\right]\\
	&\qquad-\frac{1}{k_{n}\Delta_{n}}\sum_{j=1}^{k_{n}-2}A\left(\lm{Y}{j-1},\alpha\right)^{-1}\left[b\left(\lm{Y}{j-1},\beta^{\star}\right),e_{j,n}\right]\\
	&\qquad+\frac{1}{k_{n}}\sum_{j=1}^{k_{n}-2}A\left(\lm{Y}{j-1},\alpha\right)^{-1}\left[b\left(\lm{Y}{j-1},\beta\right),b\left(X_{j\Delta_{n}},\alpha^{\star}\right)-b\left(\lm{Y}{j-1},\alpha^{\star}\right)\right]\\
	&\qquad-\frac{1}{k_{n}}\sum_{j=1}^{k_{n}-2}A\left(\lm{Y}{j-1},\alpha\right)^{-1}\left[b\left(\lm{Y}{j-1},\beta^{\star}\right),b\left(X_{j\Delta_{n}},\alpha^{\star}\right)-b\left(\lm{Y}{j-1},\alpha^{\star}\right)\right],\\
	\mathbb{Y}_{2,n}^{\mathrm{ML}\left(\dagger\right)}\left(\beta;\vartheta^{\star}\right)&=-\frac{1}{2k_{n}}\sum_{j=1}^{k_{n}-2}A\left(\lm{Y}{j-1},\hat{\alpha}_{n}\right)^{-1}
	\left[\left(b\left(\lm{Y}{j-1},\beta\right)-b\left(\lm{Y}{j-1},\beta^{\star}\right)\right)^{\otimes 2}\right].
\end{align*}
It is easy to obtain
\begin{align*}
	\sup_{n\in\mathbf{N}}\mathbf{E}_{\theta^{\star}}\left[\sup_{\theta\in\Theta}\left|M_{2,n}^{\mathrm{ML}\left(\dagger\right)}\right|^{p}\right]&\le C\left(p\right)\left(k_{n}\Delta_{n}\right)^{-p/2}
\end{align*}
using $L^{p}$-boundedness of $\sqrt{k_{n}}\left(\hat{\alpha}_{n}-\alpha^{\star}\right)$, Burkholder's inequality and Sobolev's one, and
\begin{align*}
	\sup_{n\in\mathbf{N}}\mathbf{E}_{\theta^{\star}}\left[\sup_{\theta\in\Theta}\left|R_{2,n}^{\mathrm{ML}\left(\dagger\right)}\right|^{p}\right]&\le C\left(p\right)\Delta_{n}^{p/2}
\end{align*}
because of Lemma \ref{ExpansionLM}. Let us define
\begin{align*}
	\mathbb{Y}_{2,n}^{\mathrm{ML}\left(\ddagger\right)}\left(\beta;\vartheta^{\star}\right)&=-\frac{1}{2k_{n}}\sum_{j=1}^{k_{n}-2}A\left(X_{j\Delta_{n}},\alpha^{\star}\right)^{-1}
	\left[\left(b\left(X_{j\Delta_{n}},\beta\right)-b\left(X_{j\Delta_{n}},\beta^{\star}\right)\right)^{\otimes 2}\right],
\end{align*}
and then because of $L^{p}$-boundedness of $\sqrt{k_{n}}\left(\hat{\alpha}_{n}-\alpha^{\star}\right)$, and Lemma \ref{ApproxSumLM}, we obtain
\begin{align*}
	k_{n}^{\epsilon_{1}}\left\|\sup_{\beta\in\Theta_{2}}\left|\mathbb{Y}_{2,n}^{\mathrm{ML}\left(\dagger\right)}\left(\beta;\vartheta^{\star}\right)-\mathbb{Y}_{2,n}^{\mathrm{ML}\left(\ddagger\right)}\left(\beta;\vartheta^{\star}\right)\right|\right\|_{p}\to0.
\end{align*}
Then $L^{p}$-boundedness of $\sup_{\beta\in\Theta_{2}}\left(k_{n}\Delta_{n}\right)^{\epsilon_{1}}\left|\mathbb{Y}_{2,n}^{\mathrm{ML}\left(\ddagger\right)}\left(\beta;\vartheta^{\star}\right)-\mathbb{Y}_{2}\left(\beta;\vartheta^{\star}\right)\right|$ is obtained by the discussion in Remark \ref{RemarkCentred} and it verifies (b).
\end{proof}

\begin{lemma}\label{betaInfo}
\begin{enumerate}
\item[(a)] For every $M_{3}>0$,
\begin{align*}
\sup_{n\in\mathbf{N}}\mathbf{E}_{\theta^{\star}}\left[\left(\left(k_{n}\Delta_{n}\right)^{-1}\sup_{\beta\in\Theta_{2}}\left|\partial_{\beta}^{3}\mathbb{H}_{2,n}\left(\hat{\alpha}_{n},\beta\right)\right|\right)^{M_{3}}\right]&<\infty,\\
\sup_{n\in\mathbf{N}}\mathbf{E}_{\theta^{\star}}\left[\left(\left(k_{n}\Delta_{n}\right)^{-1}\sup_{\beta\in\Theta_{2}}\left|\partial_{\beta}^{3}\mathbb{H}_{2,n}\left(\tilde{\alpha}_{n},\beta\right)\right|\right)^{M_{3}}\right]&<\infty.
\end{align*}
\item[(b)] Let $\epsilon_{1}=\epsilon_{0}/2$. Then for every $M_{4}>0$,
\begin{align*}
\sup_{n\in\mathbf{N}}\mathbf{E}_{\theta^{\star}}\left[\left(\left(k_{n}\Delta_{n}\right)^{\epsilon_{1}}\left|\Gamma_{2,n}^{\mathrm{ML}}\left(\beta^{\star};\vartheta^{\star}\right)-\Gamma_{2}\left(\vartheta^{\star}\right)\right|\right)^{M_{4}}\right]&<\infty,\\
\sup_{n\in\mathbf{N}}\mathbf{E}_{\theta^{\star}}\left[\left(\left(k_{n}\Delta_{n}\right)^{\epsilon_{1}}\left|\Gamma_{2,n}^{\mathrm{Bayes}}\left(\beta^{\star};\vartheta^{\star}\right)-\Gamma_{2}\left(\vartheta^{\star}\right)\right|\right)^{M_{4}}\right]&<\infty.
\end{align*}
\end{enumerate}
\end{lemma}

\begin{proof}
With respect to (a), we have for all $\alpha\in\Theta_{1}$ and $\beta\in\Theta_{2}$,
\begin{align*}
	&\frac{1}{k_{n}\Delta_{n}}\left|\partial_{\beta}^{3}\mathbb{H}_{2,n}\left(\alpha,\beta\right)\right|\\
	&=\frac{1}{k_{n}\Delta_{n}}\sum_{j=1}^{k_{n}-2}\left|\partial_{\beta}^{2}\left(A\left(\lm{Y}{j-1},\alpha\right)\left[\lm{Y}{j+1}-\lm{Y}{j}-\Delta_{n}b\left(\lm{Y}{j-1},\beta\right), 
	\Delta_{n}\partial_{\beta}b\left(\lm{Y}{j-1},\beta\right)^{T}\right]\right)\right|\\
	&=\frac{1}{k_{n}}\sum_{j=1}^{k_{n}-2}\left|\partial_{\beta}^{2}\left(A\left(\lm{Y}{j-1},\alpha\right)\left[\lm{Y}{j+1}-\lm{Y}{j}-\Delta_{n}b\left(\lm{Y}{j-1},\beta\right), 
	\partial_{\beta}b\left(\lm{Y}{j-1},\beta\right)^{T}\right]\right)\right|\\
	&\le \frac{1}{k_{n}}\sum_{j=1}^{k_{n}-2}C\left(1+\left|\lm{Y}{j-1}\right|+\left|\lm{Y}{j}\right|+\left|\lm{Y}{j+1}\right|\right)^{C}.
\end{align*}
Hence the evaluation of (a) can be obtained because of the integrability of $\left\{\lm{Y}{j}\right\}_{j=0,\ldots,k_{n}-1}$.

For (b), it is quite analogous to the (b) in Lemma \ref{betaScore}.
\end{proof}

\begin{proof}[Proof of Theorem \ref{mainthm}]
The first polynomial-type large deviation inequality has already been shown in Proposition \ref{MomentAlpha}, and the second and third ones are also the consequence of Lemma \ref{betaScore}, Lemma \ref{betaInfo} above and Theorem 3 in \cite{Yoshida-2011}. This result, Lemma \ref{MomentLambda} and convergence in distribution shown by \cite{Nakakita-Uchida-2018a} complete the proof for convergence of moments with respect to the adaptive ML-type estimator.

Let us define the following statistical random fields, for all $u_{0}\in\R^{d\left(d+1\right)/2}$ and $n\in\mathbf{N}$ such that $\theta_{\varepsilon}^{\star}+n^{-1/2}u_{0}\in\Theta_{\varepsilon}$,
\begin{align*}
	\mathbb{H}_{0,n}\left(\theta_{\varepsilon}\right)&:=-\frac{1}{2}\sum_{i=1}^{n-1}\left|\frac{1}{2}Z_{i+1}-\theta_{\varepsilon}\right|^{2},\\
	\mathbb{Z}_{0,n}\left(u_{0};\theta_{\varepsilon}^{\star}\right)&:=\exp\left(\mathbb{H}_{0,n}\left(\theta_{\varepsilon}^{\star}+n^{-1/2}u_{0}\right)-\mathbb{H}_{0,n}\left(\theta_{\varepsilon}^{\star}\right)\right),
\end{align*}
where $\theta_{\varepsilon}=\mathrm{vech}\Lambda$ and $Z_{i+1}=\mathrm{vech}\left\{\left(Y_{\left(i+1\right)h_{n}}-Y_{ih_{n}}\right)^{\otimes 2}\right\}$. Note that $\hat{\theta}_{\varepsilon,n}$ maximises $\mathbb{H}_{0,n}$. Now we prove the convergence in distribution such that for all $R>0$,
\begin{align*}
	&\left[\begin{matrix}
	\mathbb{Z}_{0,n}\left(u_{0};\theta_{\varepsilon}^{\star}\right), & \mathbb{Z}_{1,n}^{\tau}\left(u_{1};\hat{\Lambda}_{n},\alpha^{\star}\right), &\mathbb{Z}_{2,n}\left(u_{2};\tilde{\alpha}_{n},\beta^{\star}\right)
	\end{matrix}
	\right]\\
	&\qquad\overset{d}{\to}
	\left[\begin{matrix}
		\mathbb{Z}_{0}\left(u_{0};\theta_{\varepsilon}^{\star}\right),
		& \mathbb{Z}_{1}^{\tau}\left(u_{1};\Lambda_{\star},\alpha^{\star}\right),
		& \mathbb{Z}_{2}\left(u_{2};\alpha^{\star},\beta^{\star}\right)
	\end{matrix}
	\right]\text{ in }\mathcal{C}\left(B\left(R;\R^{d\left(d+1\right)/2+m_{1}+m_{2}}\right)\right),
\end{align*}
where for $\Delta_{0}\sim N_{d\left(d+1\right)/2}\left(\mathbf{0},\mathcal{I}^{\left(1,1\right)}\left(\vartheta^{\star}\right)\right)$, $\Delta_{1}^{\tau}\sim N_{m_{1}}\left(\mathbf{0},\mathcal{I}^{\left(2,2\right),\tau}\left(\vartheta^{\star}\right)\right)$, $\Delta_{2}\sim N_{m_{2}}\left(\mathbf{0},\mathcal{I}^{\left(3,3\right)}\left(\vartheta^{\star}\right)\right)$ such that $\Delta_{0}$, $\Delta_{1}^{\tau}$ and $\Delta_{2}$ are diagonal,
\begin{align*}
	\mathbb{Z}_{0}\left(u_{0};\vartheta^{\star}\right)&:=
	\exp\left(\Delta_{0}\left[u_{0}\right]-\left|u_{0}\right|^{2}\right),\\
	\mathbb{Z}_{1}^{\tau}\left(u_{1};\Lambda_{\star},\alpha^{\star}\right)&:=\exp\left(\Delta_{1}^{\tau}\left[u_{1}\right]-\Gamma_{1}^{\tau}\left(\vartheta^{\star}\right)\left[u_{1}^{\otimes2}\right]\right),\\
	\mathbb{Z}_{2}\left(u_{2};\alpha^{\star},\beta^{\star}\right)&:=\exp\left(\Delta_{2}\left[u_{2}\right]-\Gamma_{2}\left(\vartheta^{\star}\right)\left[u_{2}^{\otimes2}\right]\right),
\end{align*}
and $\mathcal{C}\left(B\left(R;\R^{m}\right)\right)$ is a metric space of continuous functions on the closed ball such that $B\left(R;\R^{m}\right)=\left\{u\in\R^{m};\left|u\right|\le R\right\}$, whose norm is defined as the supreme one. To prove it, it is sufficient to show the finite-dimensional convergence of 
\begin{align*}
	&\left[\begin{matrix}
	\log\mathbb{Z}_{0,n}\left(u_{0};\theta_{\varepsilon}^{\star}\right), & \log\mathbb{Z}_{1,n}^{\tau}\left(u_{1};\hat{\Lambda}_{n},\alpha^{\star}\right), &\log\mathbb{Z}_{2,n}\left(u_{2};\tilde{\alpha}_{n},\beta^{\star}\right)
	\end{matrix}
	\right]\\
	&\qquad\overset{d}{\to}
	\left[\begin{matrix}
	\log\mathbb{Z}_{0}\left(u_{0};\theta_{\varepsilon}^{\star}\right),
	& \log\mathbb{Z}_{1}^{\tau}\left(u_{1};\Lambda_{\star},\alpha^{\star}\right),
	& \log\mathbb{Z}_{2}\left(u_{2};\alpha^{\star},\beta^{\star}\right)
	\end{matrix}
	\right],
\end{align*}
and the tightness of $\left\{\log\mathbb{Z}_{0,n}\left(u_{0}\right)|_{C(B(R))};n\in\mathbf{N}\right\}$, $\left\{\log\mathbb{Z}_{1,n}^{\tau}\left(u_{1}\right)|_{C(B(R))};n\in\mathbf{N}\right\}$, and $\left\{\log\mathbb{Z}_{2,n}\left(u_{3}\right)|_{C(B(R))};n\in\mathbf{N}\right\}$. The finite-dimensional convergence is a simple consequence of \cite{Nakakita-Uchida-2018a}, and the tightness can be obtained if we can show
\begin{align*}
	\sup_{n\in\mathbf{N}}\mathbf{E}_{\theta^{\star}}\left[\sup_{u_{0}\in B\left(R;\R^{d\left(d+1\right)/2}\right)}\left|\partial_{u_{0}}\log\mathbb{Z}_{0,n}\left(u_{0};\theta_{\varepsilon}^{\star}\right)\right|\right]&<\infty,\\
	\sup_{n\in\mathbf{N}}\mathbf{E}_{\theta^{\star}}\left[\sup_{u_{1}\in B\left(R;\R^{m_{1}}\right)}\left|\partial_{u_{1}}\log\mathbb{Z}_{1,n}^{\tau}\left(u_{1};\hat{\Lambda}_{n},\alpha^{\star}\right)\right|\right]&<\infty,\\
	\sup_{n\in\mathbf{N}}\mathbf{E}_{\theta^{\star}}\left[\sup_{u_{2}\in B\left(R;\R^{m_{2}}\right)}\left|\partial_{u_{2}}\log\mathbb{Z}_{2,n}\left(u_{2};\tilde{\alpha}_{n},\beta^{\star}\right)\right|\right]&<\infty,
\end{align*}
as \citet{Ogihara-Yoshida-2011} or \citet{Yoshida-2011}. We have the first evaluation for the simple computation, and the rest ones by Lemma \ref{alphaScore}, Lemma \ref{alphaInfo}, Lemma \ref{betaScore} and Lemma \ref{betaInfo}. Hence we obtain the convergences in distribution in $\mathcal{C}\left(B\left(R;\R^{d\left(d+1\right)/2+m_{1}+m_{2}}\right)\right)$.

Finally it is necessary to show the following evaluations for the proof utilising Theorem 10 in \cite{Yoshida-2011}: there exists $\delta_{1}>0$ and $\delta_{2}>0$ such that
\begin{align*}
	\sup_{n\in\mathbf{N}}\mathbf{E}_{\theta^{\star}}\left[\left(\int_{u_{1}:\left|u_{1}\right|\le \delta_{1}}\mathbb{Z}_{1,n}^{\tau}\left(u_{1};\hat{\Lambda}_{n},\alpha^{\star}\right)\mathrm{d}u_{1}\right)^{-1}\right]<\infty,\\
	\sup_{n\in\mathbf{N}}\mathbf{E}_{\theta^{\star}}\left[\left(\int_{u_{2}:\left|u_{2}\right|\le \delta_{2}}\mathbb{Z}_{2,n}\left(u_{2};\tilde{\alpha}_{n},\beta^{\star}\right)\mathrm{d}u_{2}\right)^{-1}\right]<\infty.
\end{align*}
Because of the Lemma 2 in \cite{Yoshida-2011}, it is sufficient to show that for some $p>d$, $\delta>0$ and $C>0$,
\begin{align*}
	\sup_{n\in\mathbf{N}}\mathbf{E}_{\theta^{\star}}\left[\left|\log\mathbb{Z}_{1,n}^{\tau}\left(u_{1};\hat{\Lambda}_{n},\alpha^{\star}\right)\right|^{p}\right]\le C\left|u_{1}\right|^{p},\quad
	\sup_{n\in\mathbf{N}}\mathbf{E}_{\theta^{\star}}\left[\left|\log\mathbb{Z}_{2,n}\left(u_{2};\tilde{\alpha}_{n},\beta^{\star}\right)\right|^{p}\right]\le C\left|u_{2}\right|^{p},
\end{align*}
for all $u_{1}$, $u_{2}$ satisfying $\left|u_{1}\right|+\left|u_{2}\right|\le \delta$, and actually it is easily obtained by Lemma \ref{alphaScore}, Lemma \ref{alphaInfo}, Lemma \ref{betaScore} and Lemma \ref{betaInfo}.
These results above lead to the following convergences because of Theorem 10 in \cite{Yoshida-2011}:
\begin{align*}
&\left[\begin{matrix}
\mathbb{Z}_{0,n}\left(u_{0};\theta_{\varepsilon}^{\star}\right), & \int f_{1}\left(u_{1}\right)\mathbb{Z}_{1,n}^{\tau}\left(u_{1};\hat{\Lambda}_{n},\alpha^{\star}\right)\mathrm{d}u_{1}, &\int f_{2}\left(u_{2}\right)\mathbb{Z}_{2,n}\left(u_{2};\tilde{\alpha}_{n},\beta^{\star}\right)\mathrm{d}u_{2}
\end{matrix}
\right]\\
&\qquad\overset{d}{\to}
\left[\begin{matrix}
\mathbb{Z}_{0}\left(u_{0};\theta_{\varepsilon}^{\star}\right),
& \int f_{1}\left(u_{1}\right)\mathbb{Z}_{1}^{\tau}\left(u_{1};\Lambda_{\star},\alpha^{\star}\right)\mathrm{d}u_{1},
& \int f_{2}\left(u_{2}\right)\mathbb{Z}_{2}\left(u_{2};\alpha^{\star},\beta^{\star}\right)\mathrm{d}u_{2}
\end{matrix}
\right]\\
&\hspace{2cm}\text{ in }\mathcal{C}\left(B\left(R;\R^{d\left(d+1\right)/2}\right)\right),
\end{align*}
for the functions $f_{1}$ and $f_{2}$ of at most polynomial growth, and the continuous mapping theorem verifies
\begin{align*}
	&\left[\begin{matrix}
	\sqrt{n}\left(\hat{\theta}_{\varepsilon,n}-\theta_{\varepsilon}^{\star}\right),  \sqrt{k_{n}}\left(\tilde{\alpha}_{n}-\alpha^{\star}\right), &\sqrt{T_n}\left(\tilde{\beta}_{n}-\beta^{\star}\right)
	\end{matrix}
	\right]\\
	&\qquad\overset{d}{\to}
	\left[\begin{matrix}
	\zeta_{0},
	& \zeta_{1}^{\tau},
	& \zeta_{2}
	\end{matrix}
	\right].
\end{align*}
Moreover, in a similar way as in the proof of Theorem 8 in \cite{Yoshida-2011}, one has that
for every $p>0$,
\begin{align*}
	\sup_{n\in\mathbf{N}} \mathbf{E}_{\theta^{\star}}\left[\left|\sqrt{T_n}(\tilde{\beta}_{n}-\beta^{\star})\right|^{p}\right]<\infty,
\end{align*}
which completes the proof.
\end{proof}

\section*{Acknowledgement}
This work 
was partially supported by 
JST CREST,
JSPS KAKENHI Grant Number %JP24300107, 
JP17H01100 
and Cooperative Research Program
of the Institute of Statistical Mathematics.

\bibliography{bib190401}
\bibliographystyle{apalike}

\end{document}